\documentclass[oneside,french,english]{amsart}
\usepackage[T1]{fontenc}
\usepackage[latin9]{inputenc}
\usepackage{geometry}
\usepackage{babel}
\makeatletter
\addto\extrasfrench{%
   \providecommand{\fg}{\ifdim\lastskip>\z@\unskip\fi~\frqq}%
}

\makeatother
\usepackage{mathrsfs}
\usepackage{amsthm}
\usepackage{amstext}
\usepackage{amssymb}
\usepackage{stmaryrd}
\usepackage{graphicx}
\usepackage{esint}
\usepackage{xargs}[2008/03/08]
\usepackage[unicode=true,pdfusetitle,
 bookmarks=true,bookmarksnumbered=false,bookmarksopen=false,
 breaklinks=false,pdfborder={0 0 1},backref=false,colorlinks=false]
 {hyperref}
\usepackage{breakurl}

\makeatletter
\numberwithin{equation}{section}
\numberwithin{figure}{section}
\theoremstyle{plain}
\newtheorem{thm}{\protect\theoremname}[section]
  \theoremstyle{definition}
  \newtheorem{defn}[thm]{\protect\definitionname}
  \theoremstyle{remark}
  \newtheorem{rem}[thm]{\protect\remarkname}
  \theoremstyle{plain}
  \newtheorem{prop}[thm]{\protect\propositionname}
  \theoremstyle{plain}
  \newtheorem{lem}[thm]{\protect\lemmaname}
  \theoremstyle{plain}
  \newtheorem{cor}[thm]{\protect\corollaryname}
  \theoremstyle{plain}
  \newtheorem{fact}[thm]{\protect\factname}

\makeatother

  \addto\captionsenglish{\renewcommand{\corollaryname}{Corollary}}
  \addto\captionsenglish{\renewcommand{\definitionname}{Definition}}
  \addto\captionsenglish{\renewcommand{\factname}{Fact}}
  \addto\captionsenglish{\renewcommand{\lemmaname}{Lemma}}
  \addto\captionsenglish{\renewcommand{\propositionname}{Proposition}}
  \addto\captionsenglish{\renewcommand{\remarkname}{Remark}}
  \addto\captionsenglish{\renewcommand{\theoremname}{Theorem}}
  \addto\captionsfrench{\renewcommand{\corollaryname}{Corollaire}}
  \addto\captionsfrench{\renewcommand{\definitionname}{Définition}}
  \addto\captionsfrench{\renewcommand{\factname}{Fait}}
  \addto\captionsfrench{\renewcommand{\lemmaname}{Lemme}}
  \addto\captionsfrench{\renewcommand{\propositionname}{Proposition}}
  \addto\captionsfrench{\renewcommand{\remarkname}{Remarque}}
  \addto\captionsfrench{\renewcommand{\theoremname}{Théorème}}
  \providecommand{\corollaryname}{Corollary}
  \providecommand{\definitionname}{Definition}
  \providecommand{\factname}{Fact}
  \providecommand{\lemmaname}{Lemma}
  \providecommand{\propositionname}{Proposition}
  \providecommand{\remarkname}{Remark}
\providecommand{\theoremname}{Theorem}

\begin{document}

\global\long\def\eps{\varepsilon}

\global\long\def\tt#1{\mathtt{#1}}

\global\long\def\tx#1{\mathrm{#1}}

\global\long\def\cal#1{\mathcal{#1}}

\global\long\def\scrib#1{\mathscr{#1}}

\global\long\def\frak#1{\mathfrak{#1}}

\global\long\def\math#1{{\displaystyle #1}}

\global\long\def\pain{{\displaystyle \left(P_{j}\right)}}


\global\long\def\pare#1{\Big({\displaystyle #1}\Big)}

\global\long\def\acc#1{\left\{  #1\right\}  }

\global\long\def\cro#1{\left[#1\right]}

\newcommandx\dd[1][usedefault, addprefix=\global, 1=]{\tx d#1}

\global\long\def\DD#1{\tx D#1}

\global\long\def\ppp#1#2{\frac{\partial#1}{\partial#2}}

\global\long\def\ddd#1#2{\frac{\tx d#1}{\tx d#2}}

\global\long\def\norm#1{\left\Vert #1\right\Vert }

\global\long\def\abs#1{\left|#1\right|}

\global\long\def\ps#1{\left\langle #1\right\rangle }

\global\long\def\crocro#1{\left\llbracket #1\right\rrbracket }

\global\long\def\ord#1{\tx{ord}\left(#1\right)}


\global\long\def\wc{\mathbb{C}}

\global\long\def\ww#1{\mathbb{#1}}

\global\long\def\wr{\mathbb{R}}

\global\long\def\wn{\mathbb{N}}

\global\long\def\wn{\mathbb{Z}}

\global\long\def\wq{\mathbb{Q}}

\global\long\def\wcc{\left(\ww C^{2},0\right)}

\global\long\def\wccc{\left(\ww C^{3},0\right)}

\global\long\def\wcn{\left(\ww C^{n},0\right)}

\global\long\def\cn{\ww C^{n},0}


\global\long\def\germ#1{\ww C\left\{  #1\right\}  }

\global\long\def\germinv#1{\ww C\left\{  #1\right\}  ^{\times} }

\global\long\def\form#1{\ww C\left\llbracket #1\right\rrbracket }

\global\long\def\forminv#1{\ww C\left\llbracket #1\right\rrbracket ^{\times}}

\global\long\def\maxid{\mathfrak{m}}

\newcommandx\jet[1][usedefault, addprefix=\global, 1=k]{\mbox{\ensuremath{\tx J}}_{#1}}

\global\long\def\hol#1{\mathcal{O}\left(#1\right)}

\newcommandx\formdiff[1][usedefault, addprefix=\global, 1=]{\widehat{\Omega^{#1}}}


\global\long\def\pp#1{\frac{\partial}{\partial#1}}

\newcommandx\der[2][usedefault, addprefix=\global, 1=\mathbf{x}, 2=]{\tx{Der}_{#2}\left(\form{#1}\right)}

\newcommandx\vf[1][usedefault, addprefix=\global, 1={\ww C^{3},0}]{\chi\left(#1\right)}

\newcommandx\fvf[1][usedefault, addprefix=\global, 1={\ww C^{3},0}]{\widehat{\chi}\left(#1\right)}

\newcommandx\sfvf[1][usedefault, addprefix=\global, 1={\ww C^{3},0}]{\widehat{\chi}_{\omega}\left(#1\right)}

\newcommandx\fisot[2][usedefault, addprefix=\global, 1=\tx{fib}]{\widehat{\tx{Isot}}_{#1}\left(#2\right)}

\global\long\def\isotsect#1#2#3{\tx{Isot}_{\tx{fib}}\left(#1;\cal S_{#2,#3};\tx{Id}\right)}

\global\long\def\ynorm{Y_{\tx{norm}}}

\global\long\def\lie#1{\cal L_{\math{#1}}}


\newcommandx\fdiff[1][usedefault, addprefix=\global, 1={\ww C^{3},0}]{\tx{Diff}_{\tx{fib}}\left(#1\right)}

\newcommandx\diff[1][usedefault, addprefix=\global, 1={\ww C^{3},0}]{\tx{Diff}\left(#1\right)}

\newcommandx\diffid[1][usedefault, addprefix=\global, 1={\ww C^{3},0}]{\tx{Diff}\left(#1;\tx{Id}\right)}

\newcommandx\fdiffid[1][usedefault, addprefix=\global, 1={\ww C^{3},0}]{\tx{Diff}_{\tx{fib}}\left(#1;\tx{Id}\right)}

\newcommandx\diffsect[2][usedefault, addprefix=\global, 1=\theta, 2=\eta]{\tx{Diff}_{\tx{fib}}\left(\cal S_{#1,#2};\tx{Id}\right)}

\newcommandx\diffform[1][usedefault, addprefix=\global, 1={\ww C^{3},0}]{\widehat{\tx{Diff}}\left(#1\right)}

\newcommandx\fdiffform[1][usedefault, addprefix=\global, 1={\ww C^{3},0}]{\widehat{\tx{Diff}}_{\tx{fib}}\left(#1\right)}

\newcommandx\fdiffformid[1][usedefault, addprefix=\global, 1={\ww C^{3},0}]{\widehat{\tx{Diff}}_{\tx{fib}}\left(#1;\tx{Id}\right)}

\newcommandx\gid[1][usedefault, addprefix=\global, 1=k]{\cal G_{\tx{Id}}^{\left(#1\right)}}

\newcommandx\diffformid[1][usedefault, addprefix=\global, 1={\ww C^{3},0}]{\widehat{\tx{Diff}}\left(#1;\tx{Id}\right)}

\newcommandx\sdiff[1][usedefault, addprefix=\global, 1={\ww C^{3},0}]{\tx{Diff}_{\omega}\left(#1\right)}

\newcommandx\sdiffid[1][usedefault, addprefix=\global, 1={\ww C^{3},0}]{\tx{Diff}_{\omega}\left(#1;\tx{Id}\right)}

\newcommandx\sdiffform[1][usedefault, addprefix=\global, 1={\ww C^{3},0}]{\widehat{\tx{Diff}}_{\omega}\left(#1\right)}

\newcommandx\sdiffformid[1][usedefault, addprefix=\global, 1={\ww C^{3},0}]{\widehat{\tx{Diff}}_{\omega}\left(#1;\tx{Id}\right)}


\global\long\def\sect#1#2#3{S\left(#1,#2,#3\right)}

\global\long\def\germsect#1#2{\cal S_{#1,#2}}

\global\long\def\asympsect#1#2{\cal{AS}_{#1,#2}}

\global\long\def\fol#1{\mathcal{F}_{#1}}


\global\long\def\rrel{\mathcal{R}}

\global\long\def\surj{\twoheadrightarrow}

\global\long\def\inj{\hookrightarrow}

\global\long\def\bij{\simeq}

\global\long\def\quotient#1#2{\bigslant{#1}{#2}}


\global\long\def\res#1{\tx{res}\left(#1\right)}

\global\long\def\param{\cal P}

\global\long\def\po{\cal P_{0}}

\global\long\def\pfib{\cal P_{\tx{fib}}}

\global\long\def\pw{\cal P_{\omega}}

\global\long\def\porb{\cal P_{\tx{orb}}}

\global\long\def\ls#1{\Gamma_{#1\lambda}}

\global\long\def\lsp#1{\Gamma'_{#1\lambda}}

\global\long\def\lsr#1#2{\Gamma_{#1\lambda}\left(#2\right)}

\global\long\def\sn{\mathbb{\cal{SN}}}

\global\long\def\sndiag{\mathbb{\cal{SN}}_{\tx{diag}}}

\global\long\def\sndiagnd{\mathbb{\cal{SN}}_{\tx{diag,nd}}}

\global\long\def\snodiag{\cal{SN}_{\tx{diag},0}}

\global\long\def\sns{\mathbb{\cal{SN}}_{\omega}}

\global\long\def\snsdiag{\mathbb{\cal{SN}}_{\mbox{\ensuremath{\tx{diag}}},\omega}}

\global\long\def\snsnd{\mathbb{\cal{SN}}_{\tx{snd}}}

\global\long\def\sno{\cal{SN}_{0}}

\global\long\def\snfib{\cal{SN}_{\tx{fib}}}

\global\long\def\snofib{\cal{SN}_{\tx{fib},0}}

\global\long\def\snorb{\cal{SN}_{\tx{orb}}}

\global\long\def\snoid{\cal{SN}_{\tx{Id},0}}

\global\long\def\snoorb{\cal{SN}_{0,\tx{orb}}}

\global\long\def\fsn{\mathbb{\widehat{\cal{SN}}}}

\global\long\def\fsns{\widehat{\mathbb{\cal{SN}}}_{\omega}}

\global\long\def\fsndiag{\mathbb{\widehat{\cal{SN}}}_{\tx{diag}}}

\global\long\def\fsnnd{\mathbb{\widehat{\cal{SN}}}_{\tx{nd}}}

\global\long\def\fsnfibnd{\mathbb{\widehat{\cal{SN}}}_{\tx{fib,nd}}}

\global\long\def\fsndiagnd{\mathbb{\widehat{\cal{SN}}}_{\tx{diag,nd}}}

\global\long\def\fsno{\cal{\widehat{SN}}_{*}}

\global\long\def\fsnfib{\cal{\widehat{SN}}_{\tx{fib}}}

\global\long\def\fsnorb{\widehat{\cal{SN}}_{\tx{orb}}}

\global\long\def\fsnoorb{\widehat{\cal{SN}}_{\tx{orb},*}}

\global\long\def\mfibid{\mathscr{M}_{\tx{fib,Id}}}

\global\long\def\mfib{\mathscr{M}_{\tx{fib}}}

\global\long\def\mspaceid{\mathscr{M}_{\tx{Id}}}

\global\long\def\mspace{\mathscr{M}}

\global\long\def\ms{\mathscr{M}_{\omega}}

\newcommand{\bigslant}[2]{{\raisebox{.3em}{$#1$}\left/\raisebox{-.3em}{$#2$}\right.}}

\title[Analytic classification of doubly-resonant saddle-nodes]{Doubly-resonant saddle-nodes in $\left(\mathbb{C}^{3},0\right)$
and the fixed singularity at infinity in the Painlevé equations. Part
III: local analytic classification}

\author{Amaury Bittmann}

\address{IRMA, Université de Strasbourg, 7 rue René Descartes, 67084 Strasbourg
Cedex, France }

\email{\href{mailto:bittmann@math.unistra.fr}{bittmann@math.unistra.fr}}
\begin{abstract}
In this work which follows directly \cite{bittmann1,bittmann2}, we
consider analytic singular vector fields in $\ww C^{3}$ with an isolated
and doubly-resonant singularity of saddle-node type at the origin.
Such vector fields come from irregular two-dimensional differential
systems with two opposite non-zero eigenvalues, and appear for instance
when studying the irregular singularity at infinity in Painlevé equations
${\displaystyle \pain_{{\scriptstyle j=I\dots V}}}$, for generic
values of the parameters. Under suitable assumptions, we provide an
analytic classification under the action of fibered diffeomorphisms,
based on the study of the \emph{Stokes diffeomorphisms} obtained by
comparing consecutive sectorial normalizing maps \emph{à la} Martinet-Ramis~/~Stolovitch
\cite{MR82,MR83,Stolo}. These normalizing maps over sectorial domains
are obtained in the main theorem of \cite{bittmann2}, which is analogous
to the classical one due to Hukuhara-Kimura-Matuda \cite{HKM} for
saddle-nodes in $\ww C^{2}$. We also prove that these maps are in
fact the Gevrey-1 sums of the formal normalizing map, the existence
of which has been proved in \cite{bittmann1}. 
\end{abstract}


\maketitle

\section{Introduction}

~

As in \cite{bittmann1,bittmann2}, we consider (germs of) singular
vector fields $Y$ in $\ww C^{3}$ which can be written in appropriate
coordinates $\left(x,\mathbf{y}\right):=\left(x,y_{1},y_{2}\right)$
as 
\begin{eqnarray}
Y & = & x^{2}\pp x+\Big(-\lambda y_{1}+F_{1}\left(x,\mathbf{y}\right)\Big)\pp{y_{1}}+\Big(\lambda y_{2}+F_{2}\left(x,\mathbf{y}\right)\Big)\pp{y_{2}}\,\,\,\,\,,\label{eq: intro}
\end{eqnarray}
where $\lambda\in\ww C^{*}$ and $F_{1},\, F_{2}$ are germs of holomorphic
function in $\left(\ww C^{3},0\right)$ of homogeneous valuation (order)
at least two. They represent irregular two-dimensional differential
systems having two opposite non-zero eigenvalues and a vanishing third
eigenvalue. These we call doubly-resonant vector fields of saddle-node
type (or simply \textbf{doubly-resonant saddle-nodes}). For a historical
context, a presentation of the main motivations (the study of the
irregular singularity at infinity in Painlevé equations ${\displaystyle \pain_{{\scriptstyle j=I\dots V}}}$),
and a review of some results linked with this study, we refer to \cite{bittmann1}.

\bigskip{}

Several authors studied the problem of convergence of formal transformations
putting vector fields as in $\left(\mbox{\ref{eq: intro}}\right)$
into ``normal forms''. Shimomura, improving on a result of Iwano
\cite{Iwano}, shows in \cite{Shimo} that analytic doubly-resonant
saddle-nodes satisfying more restrictive conditions are conjugate
(formally and over sectors) to vector fields of the form
\begin{eqnarray*}
 &  & x^{2}\pp x+\left(-\lambda+a_{1}x\right)y_{1}\pp{y_{1}}+\left(\lambda+a_{2}x\right)y_{2}\pp{y_{2}}\,\,\,\,
\end{eqnarray*}
\emph{via} a diffeomorphism whose coefficients have asymptotic expansions
as $x\rightarrow0$ in sectors of opening greater than $\pi$. 

Stolovitch then generalized this result to any dimension in \cite{Stolo}.
More precisely, Stolovitch's work offers an analytic classification
of vector fields in $\ww C^{n+1}$ with an irregular singular point,
without further hypothesis on eventual additional resonance relations
between eigenvalues of the linear part. However, as Iwano and Shimomura
did, he needed to impose other assumptions, among which the condition
that the restriction of the vector field to the invariant hypersurface
$\acc{x=0}$ is a linear vector field. In \cite{DeMaesschalck}, the
authors obtain a \emph{Gevrey-1 summable} ``normal form'', though
not as simple as Stolovitch's one and not unique \emph{a priori},
but for more general kind of vector field with one zero eigenvalue.
However, the same assumption on hypersurface $\acc{x=0}$ is required
(the restriction is a linear vector field). Yet from \cite{Yoshida85}
(and later \cite{bittmann1}) stems the fact that this condition is
not met in the case of Painlevé equations ${\displaystyle \pain_{{\scriptstyle j=I\dots V}}}$. 

In comparison, we merely ask here that this restriction be orbitally
linearizable (see Definition \ref{def: asympt hamil}), \emph{i.e.}
the foliation\emph{ }induced by $Y$ on $\acc{x=0}$ (and not the
vector field $Y_{|\acc{x=0}}$ itself) be linearizable. The fact that
this condition is fulfilled by the singularities of Painlevé equations
formerly described is well-known. As discussed in Remark \ref{rem:New difficulties},
this more general context also introduces new phenomena and technical
difficulties as compared to prior classification results.

\subsection{Scope of the paper}

~

The action of local analytic~/~formal diffeomorphisms $\Psi$ fixing
the origin on local holomorphic vector fields $Y$ of type $\left(\mbox{\ref{eq: intro}}\right)$
by change of coordinates is given by
\begin{eqnarray*}
\Psi_{*}Y & := & \left(\mathrm{D}\Psi.Y\right)\circ\Psi^{-1}~.
\end{eqnarray*}
In \cite{bittmann1} we performed the formal classification of such
vector fields by exhibiting an explicit universal family of vector
fields for the action of formal changes of coordinates at $0$ (called
a family of normal forms). Such a result seems currently out of reach
in the analytic category: it is unlikely that an explicit universal
family for the action of local analytic changes of coordinates be
described anytime soon. If we want to describe the space of equivalent
classes (of germs of a doubly-resonant saddle-node under local analytic
changes of coordinates) with same formal normal form, we therefore
need to find a complete set of invariants which is of a different
nature. We call \textbf{moduli space} this quotient space and give
it a (non-trivial) presentation based on functional invariants \emph{à
la} Martinet-Ramis \cite{MR82,MR83}.

In this paper we will therefore present only the $x-$fibered local
analytic classification for vector fields of the form $\left(\mbox{\ref{eq: intro}}\right)$,
with some additional assumptions detailed further down (see Definitions
\ref{def: drsn}, \ref{def: non-deg} and \ref{def: asympt hamil}).
Importantly, these hypothesis are met in the case of Painlevé equations
mentioned above. The full analytic classification (under the action
of all local diffeomorphisms, not necessarily $x-$fibered) will be
done in a forthcoming work.

\bigskip{}

In \cite{bittmann2}, we have proved the existence of analytic sectorial
normalizing maps (over a pair of opposite ``wide'' sectors of opening
greater than $\pi$ whose union covers a full punctured neighborhood
of $\left\{ x=0\right\} $). Then we attach to each vector field a
complete set of invariants given as transition maps (over ``narrow''
sectors of opening less than $\pi$) between the sectorial normalizing
maps. Although this viewpoint has become classical since the work
of Martinet and Ramis, and has latter been generalized by Stolovitch
as already mentioned, our approach has a more geometric flavor (for
instance, we perform a precise study of the Stokes diffeomorphisms
in the \emph{space of leaves}).

As a by-product, we deduce that the normalizing sectorial diffeomorphisms
of \cite{bittmann2} are Gevrey-$1$ asymptotic to the normalizing
formal power series of \cite{bittmann1}, retrospectively proving
their $1$-summability. When the vector field additionally supports
a symplectic transverse structure (which is again the case of Painlevé
equations) we prove a theorem of analytic classification under the
action of \emph{transversally symplectic} diffeomorphisms.

\subsection{Definitions and previous results}

~

To state our main results we need to introduce some notations and
nomenclature. 
\begin{itemize}
\item For $n\in\ww N_{>0}$, we denote by $\left(\ww C^{n},0\right)$ an
(arbitrary small) open neighborhood of the origin in $\ww C^{n}$.
\item We denote by $\germ{x,\mathbf{y}}$, with $\mathbf{y}=\left(y_{1},y_{2}\right)$,
the $\ww C$-algebra of germs of holomorphic functions at the origin
of $\ww C^{3}$, and by $\germ{x,\mathbf{y}}^{\times}$ the group
of invertible elements for the multiplication (also called units),
\emph{i.e. }elements $U$ such that $U\left(0\right)\neq0$.
\item $\vf$ is the Lie algebra of germs of singular holomorphic vector
fields at the origin $\ww C^{3}$. Any vector field in $\vf$ can
be written as 
\[
Y={\displaystyle b\left(x,y_{1},y_{2}\right)\pp x+b_{1}\left(x,y_{1},y_{2}\right)\pp{y_{1}}+b_{2}\left(x,y_{1},y_{2}\right)\pp{y_{2}}}
\]
with $b,b_{1},b_{2}\in\germ{x,y_{1},y_{2}}$ vanishing at the origin.
\item $\diff$ is the group of germs of holomorphic diffeomorphisms fixing
the origin of $\ww C^{3}$. It acts on $\vf$ by conjugacy: for all
\[
\left(\Phi,Y\right)\in\diff\times\vf
\]
we define the push-forward of $Y$ by $\Phi$ by
\begin{equation}
\Phi_{*}\left(Y\right):=\left(\mbox{D}\Phi\cdot Y\right)\circ\Phi^{-1}\qquad,\label{eq: push forward intro}
\end{equation}
where $\mbox{D}\Phi$ is the Jacobian matrix of $\Phi$.
\item $\fdiff$ is the subgroup of $\diff$ of fibered diffeomorphisms preserving
the $x$-coordinate, \emph{i.e. }of the form $\left(x,\mathbf{y}\right)\mapsto\left(x,\phi\left(x,\mathbf{y}\right)\right)$.
\item We denote by $\fdiff[\ww C^{3},0,\tx{Id}]$ the subgroup of $\fdiff$
formed by diffeomorphisms tangent to the identity.
\end{itemize}
All these concepts have \emph{formal} analogues, where we only suppose
that the objects are defined with formal power series, not necessarily
convergent near the origin. 
\begin{defn}
\label{def: drsn}A \textbf{diagonal doubly-resonant saddle-node}
is a vector field $Y\in\vf$ of the form 
\begin{eqnarray}
Y & = & x^{2}\pp x+\Big(-\lambda y_{1}+F_{1}\left(x,\mathbf{y}\right)\Big)\pp{y_{1}}+\Big(\lambda y_{2}+F_{2}\left(x,\mathbf{y}\right)\Big)\pp{y_{2}}\,\,\,\,\,,\label{eq: diagonal drsn}
\end{eqnarray}
with $\lambda\in\ww C^{*}$ and $F_{1},F_{2}\in\germ{x,\mathbf{y}}$
of order at least two. We denote by $\sndiag$ the set of such vector
fields.
\end{defn}
Based on this expression, and considering the expansion 
\begin{eqnarray*}
F_{j}\left(x,\mathbf{y}\right) & = & \sum_{\substack{\mathbf{k}=\left(k_{0},k_{1},k_{2}\right)}
}F_{j,\mathbf{k}}x^{k_{0}}y_{1}^{k_{1}}y_{2}^{k_{2}}
\end{eqnarray*}
for $j=1,2$, we state:
\begin{defn}
\label{def: non-deg}The \textbf{residue} of $Y\in\sndiag$ as in
(\ref{eq: diagonal drsn}) is the complex number
\[
{\displaystyle \tx{res}\left(Y\right):=F_{1,\left(1,1,0\right)}+F_{2,\left(1,0,1\right)}}\,\,\,.
\]
We say that $Y$ is\textbf{ non-degenerate }(\emph{resp. }\textbf{strictly
non-degenerate) }if $\tx{res}\left(Y\right)\notin\ww Q_{\leq0}$ (\emph{resp.
}$\Re\left(\tx{res}\left(Y\right)\right)>0$).\end{defn}
\begin{rem}
It is obvious that there is an action of $\fdiff[\ww C^{3},0,\tx{Id}]$
on $\sndiag$. The residue is an invariant of each orbit of $\snfib$
under the action of $\fdiff[\ww C^{3},0,\tx{Id}]$ by conjugacy (it
is actually invariant by formal conjugacies, see\cite{bittmann1}).
\end{rem}
The main result of \cite{bittmann1} can now be stated as follows:
\begin{thm}
\cite{bittmann1}\label{thm: forme normalel formelle} Let $Y\in\sndiag$
be non-degenerate. Then there exists a unique formal fibered diffeomorphism
$\hat{\Phi}$ tangent to the identity such that: 
\begin{eqnarray}
\hat{\Phi}_{*}\left(Y\right) & = & x^{2}\pp x+\left(-\lambda+a_{1}x+c_{1}\left(y_{1}y_{2}\right)\right)y_{1}\pp{y_{1}}\nonumber \\
 &  & +\left(\lambda+a_{2}x+c_{2}\left(y_{1}y_{2}\right)\right)y_{2}\pp{y_{2}}\,\,\,\,,\label{eq: fibered normal form-1-2}
\end{eqnarray}
where $\lambda\in\ww C^{*}$, ${\displaystyle c_{1},c_{2}\in v\form v}$
are formal power series in $v=y_{1}y_{2}$ without constant term and
$a_{1},a_{2}\in\ww C$ are such that ${\displaystyle a_{1}+a_{2}=\tx{res}\left(Y\right)\in\ww C\backslash\ww Q_{\leq0}}$.\end{thm}
\begin{defn}
The vector field obtained in $\left(\mbox{\ref{eq: fibered normal form-1-2}}\right)$
is called the \textbf{formal normal form }of $Y$. The formal fibered
diffeomorphism $\hat{\Phi}$ is called the \textbf{formal normalizing
map} of $Y$.
\end{defn}
The above result is valid for formal objects, without considering
problems of convergence. The main result in \cite{bittmann2} states
that this formal normalizing map is analytic in sectorial domains,
under some additional assumptions that we are now going to precise.
\begin{defn}
\label{def: asympt hamil}~
\begin{itemize}
\item We say that a germ of a vector field $X$ in $\left(\ww C^{2},0\right)$
is \textbf{orbitally linear} if 
\[
X=U\left(\mathbf{y}\right)\left(\lambda_{1}y_{1}\pp{y_{1}}+\lambda_{2}y_{2}\pp{y_{2}}\right)\,\,,
\]
for some ${\displaystyle U\left(\mathbf{y}\right)\in\germ{\mathbf{y}}^{\times}}$
and $\left(\lambda_{1},\lambda_{2}\right)\in\ww C^{2}$.
\item We say that a germ of vector field $X$ in $\left(\ww C^{2},0\right)$
is analytically (\emph{resp. formally}) \textbf{orbitally linearizable}
if $X$ is analytically (\emph{resp.} formally) conjugate to an orbitally
linear vector field.
\item We say that a diagonal doubly-resonant saddle-node $Y\in\sndiag$
is \textbf{div-integrable} if $Y_{\mid\acc{x=0}}\in\vf[\ww C^{2},0]$
is (analytically) orbitally linearizable. 
\end{itemize}
\end{defn}
\begin{rem}
Alternatively we could say that the foliation associated to ${\displaystyle Y_{\mid\acc{x=0}}}$
is linearizable. Since ${\displaystyle Y_{\mid\acc{x=0}}}$ is analytic
at the origin of $\ww C^{2}$ and has two opposite eigenvalues, it
follows from a classical result of Brjuno (see \cite{Martinet}),
that $Y_{\mid\acc{x=0}}$ is analytically orbitally linearizable if
and only if it is formally orbitally linearizable.\end{rem}
\begin{defn}
We denote by $\snodiag$ the set of strictly non-degenerate diagonal
doubly-resonant saddle-nodes which are div-integrable.
\end{defn}
The main result of \cite{bittmann2} can now be stated (we refer to
section $2.$ for precise definitions on weak 1-summability)). 
\begin{thm}
\cite{bittmann2}\label{Th: Th drsn}Let $Y\in\snodiag$ and let $\hat{\Phi}$
(given by Theorem \ref{thm: forme normalel formelle}) be the unique
formal fibered diffeomorphism tangent to the identity such that 
\begin{eqnarray*}
\hat{\Phi}_{*}\left(Y\right) & = & x^{2}\pp x+\left(-\lambda+a_{1}x+c_{1}\left(y_{1}y_{2}\right)\right)y_{1}\pp{y_{1}}+\left(\lambda+a_{2}x+c_{2}\left(y_{1}y_{2}\right)\right)y_{2}\pp{y_{2}}\\
 & =: & \ynorm\,\,,
\end{eqnarray*}
where $\lambda\neq0$ and ${\displaystyle c_{1}\left(v\right),c_{2}\left(v\right)\in v\form v}$
are formal power series without constant term. Then:
\begin{enumerate}
\item the normal form $\ynorm$ is analytic \emph{(i.e. ${\displaystyle c_{1},c_{2}\in\germ v}$),
}and it also is div-integrable,\emph{ }i.e. $c_{1}+c_{2}=0$;
\item the formal normalizing map $\hat{\Phi}$ is weakly 1-summable in every
direction $\theta\neq\arg\left(\pm\lambda\right)$;
\item there exist analytic sectorial fibered diffeomorphisms $\Phi_{+}$
and $\Phi_{-}$, (asymptotically) tangent to the identity, defined
in sectorial domains of the form ${\displaystyle S_{+}\times\left(\ww C^{2},0\right)}$
and ${\displaystyle S_{-}\times\left(\ww C^{2},0\right)}$ respectively,
where 
\begin{eqnarray*}
S_{+} & := & \acc{x\in\ww C\mid0<\abs x<r\mbox{ and }\abs{\arg\left(\frac{x}{i\lambda}\right)}<\frac{\pi}{2}+\epsilon}\\
S_{-} & := & \acc{x\in\ww C\mid0<\abs x<r\mbox{ and }\abs{\arg\left(\frac{-x}{i\lambda}\right)}<\frac{\pi}{2}+\epsilon}
\end{eqnarray*}
(for any ${\displaystyle \epsilon\in\left]0,\frac{\pi}{2}\right[}$
and some $r>0$ small enough), which admit $\hat{\Phi}$ as weak Gevrey-1
asymptotic expansion in these respective domains, and which conjugate
$Y$ to $\ynorm$. Moreover $\Phi_{+}$ and $\Phi_{-}$ are the unique
such germs of analytic functions in sectorial domains (see Definition
\ref{def: sectorial diff}).
\end{enumerate}
\end{thm}
\begin{defn}
\label{def: sectorial normalizations}We call $\Phi_{+}$ and $\Phi_{-}$
the \textbf{sectorial normalizing maps} of $Y\in\snodiag$. 
\end{defn}
They are the weak 1-sums of $\hat{\Phi}$ along the respective directions
$\arg\left(i\lambda\right)$ and $\arg\left(-i\lambda\right)$. Notice
that $\Phi_{+}$ and $\Phi_{-}$ are \emph{germs of analytic sectorial
fibered diffeomorphisms}, \emph{i.e.} they are of the form
\begin{eqnarray*}
\Phi_{+}:S_{+}\times\left(\ww C^{2},0\right) & \longrightarrow & S_{+}\times\left(\ww C^{2},0\right)\\
\left(x,\mathbf{y}\right) & \longmapsto & \left(x,\Phi_{+,1}\left(x,\mathbf{y}\right),\Phi_{+,2}\left(x,\mathbf{y}\right)\right)
\end{eqnarray*}
and 
\begin{eqnarray*}
\Phi_{-}:S_{-}\times\left(\ww C^{2},0\right) & \longrightarrow & S_{-}\times\left(\ww C^{2},0\right)\\
\left(x,\mathbf{y}\right) & \longmapsto & \left(x,\Phi_{-,1}\left(x,\mathbf{y}\right),\Phi_{-,2}\left(x,\mathbf{y}\right)\right)
\end{eqnarray*}
(see section $2.$ for a precise definition of \emph{germ of analytic
sectorial fibered diffeomorphism}). The fact that they are\emph{ }also
\emph{(asymptotically) tangent to the identity }means that we have:
\[
{\displaystyle \Phi_{\pm}\left(x,\mathbf{y}\right)=\tx{Id}+\mbox{\ensuremath{\tx O}}\left(\norm{\left(x,\mathbf{y}\right)}^{2}\right)}\,\,.
\]

Another result proved in \cite{bittmann2}, is that the uniqueness
of the sectorial normalizing maps holds in fact under weaker assumptions.
\begin{prop}
\label{prop: unique normalizations}Let $\varphi_{+}$ and $\varphi_{-}$
be two germs of sectorial fibered diffeomorphisms in ${\displaystyle S_{+}\times\left(\ww C^{2},0\right)}$
and ${\displaystyle S_{-}\times\left(\ww C^{2},0\right)}$ respectively,
where $S_{+}$ and $S_{-}$ are as in Theorem \ref{Th: Th drsn},
which are (asymptotically) tangent to the identity and such that 
\[
\left(\varphi_{\pm}\right)_{*}\left(Y\right)=\ynorm\,\,.
\]
Then, they necessarily coincide with the weak 1-sums $\Phi_{+}$ and
$\Phi_{-}$ defined above. 
\end{prop}

\subsection{Main results}

~

The first main result of this paper is the following.
\begin{thm}
\label{Th: Th drsn-1}Let $Y\in\snodiag$ and let $\hat{\Phi}$ (given
by Theorem \ref{thm: forme normalel formelle}) be the unique formal
fibered diffeomorphism tangent to the identity such that 
\begin{eqnarray*}
\hat{\Phi}_{*}\left(Y\right) & = & x^{2}\pp x+\left(-\lambda+a_{1}x-c\left(y_{1}y_{2}\right)\right)y_{1}\pp{y_{1}}+\left(\lambda+a_{2}x+c\left(y_{1}y_{2}\right)\right)y_{2}\pp{y_{2}}\\
 & =: & \ynorm\,\,,
\end{eqnarray*}
where $\lambda\neq0$ and ${\displaystyle c\left(v\right)\in v\germ v}$.
Then $\hat{\Phi}$ is 1-summable (with respect to $x$) in every direction
$\theta\neq\arg\left(\pm\lambda\right)$, and $\Phi_{+},\Phi_{-}$
in Theorem \ref{Th: Th drsn} are the 1-sums of $\hat{\Phi}$ in directions
$\arg\left(i\lambda\right),\arg\left(-i\lambda\right)$ respectively.
\end{thm}
Since two analytically conjugate vector fields are also formally conjugate,
we fix now a normal form 
\[
\ynorm=x^{2}\pp x+\left(-\lambda+a_{1}x-c\left(v\right)\right)y_{1}\pp{y_{1}}+\left(\lambda+a_{2}x+c\left(v\right)\right)y_{2}\pp{y_{2}}\,\,\,\,,
\]
with $\lambda\in\ww C^{*}$, $\Re\left(a_{1}+a_{2}\right)>0$ and
$c\in v\germ v$ vanishing at the origin.
\begin{defn}
\label{def: ynorm class}We denote by ${\displaystyle \cro{\ynorm}}$
the set of germs of holomorphic doubly-resonant saddle-nodes in $\left(\ww C^{3},0\right)$
which are formally conjugate to $\ynorm$ by formal fibered diffeomorphisms
tangent to the identity, and denote by ${\displaystyle \quotient{\cro{\ynorm}}{\fdiff[\ww C^{3},0,\tx{Id}]}}$
the set of orbits of the elements in this set under the action of
${\displaystyle \fdiff[\ww C^{3},0,\tx{Id}]}$.
\end{defn}
According to Theorem \ref{Th: Th drsn}, to any ${\displaystyle Y\in{\displaystyle \cro{\ynorm}}}$
we can associate two sectorial normalizing maps $\Phi_{+},\Phi_{-}$,
which can in fact extend analytically in domains $S_{+}\times\left(\ww C^{2},0\right)$
and $S_{-}\times\left(\ww C^{2},0\right)$, where $S_{\pm}$ is an
asymptotic sector in the direction $\arg\left(\pm i\lambda\right)$
with opening $2\pi$ (see Definition \ref{def: asymptotitc sector}):
\[
\left(S_{+},S_{-}\right)\in\asympsect{\arg\left(i\lambda\right)}{2\pi}\times\asympsect{\arg\left(-i\lambda\right)}{2\pi}\,\,.
\]
Then, we consider two germs of sectorial fibered diffeomorphisms $\Phi_{\lambda},\Phi_{-\lambda}$
analytic in $S_{\lambda},S_{-\lambda}$, with 
\begin{eqnarray}
S_{\lambda} & := & S_{+}\cap S_{-}\cap\acc{\Re\left(\frac{x}{\lambda}\right)>0}\in\asympsect{\arg\left(\lambda\right)}{\pi}\label{eq: petits secteurs}\\
S_{-\lambda} & := & S_{+}\cap S_{-}\cap\acc{\Re\left(\frac{x}{\lambda}\right)<0}\in\asympsect{\arg\left(-\lambda\right)}{\pi}\,\,,\nonumber 
\end{eqnarray}
defined by: 
\[
\begin{cases}
\Phi_{\lambda}:=\left(\Phi_{+}\circ\Phi_{-}^{-1}\right)_{\mid S_{\lambda}\times\left(\ww C^{2},0\right)}\in\diffsect[\arg\left(\lambda\right)][\epsilon] & ,\,\forall\epsilon\in\left[0,\pi\right[\\
\Phi_{-\lambda}:=\left(\Phi_{-}\circ\Phi_{+}^{-1}\right)_{\mid S_{-\lambda}\times\left(\ww C^{2},0\right)}\diffsect[\arg\left(-\lambda\right)][\epsilon] & ,\,\forall\epsilon\in\left[0,\pi\right[\,\,.
\end{cases}
\]
Notice that $\Phi_{\lambda},\Phi_{-\lambda}$ are \emph{isotropies}
of $\ynorm$, \emph{i.e. }they satisfy:
\begin{eqnarray*}
\left(\Phi_{\pm\lambda}\right)_{*}\left(\ynorm\right) & = & \ynorm\,\,.
\end{eqnarray*}

\begin{defn}
\label{def: Stokes diffeo}With the above notations, we define ${\displaystyle \Lambda_{\lambda}\left(\ynorm\right)}$
${\displaystyle \left(\mbox{\emph{resp.} }\Lambda_{-\lambda}\left(\ynorm\right)\right)}$
as the group of germs of sectorial fibered isotropies of $\ynorm$,
tangent to the identity, and admitting the identity as Gevrey-1 asymptotic
expansion (see Definition \ref{def:Gevrey-1_expansion}) in sectorial
domains of the form ${\displaystyle S_{\lambda}\times\left(\ww C^{2},0\right)}$
${\displaystyle \left(resp.\, S_{-\lambda}\times\left(\ww C^{2},0\right)\right)}$,
with $S_{\pm\lambda}\in\asympsect{\arg\left(\pm\lambda\right)}{\pi}$. 

The two sectorial isotropies $\Phi_{\lambda}$ and $\Phi_{-\lambda}$
defined above are called the \textbf{Stokes diffeomorphisms} associate
to ${\displaystyle Y\in{\displaystyle \cro{\ynorm}}}$.
\end{defn}
Our second main result gives the moduli space for the analytic classification
that we are looking for.
\begin{thm}
\label{thm: espace de module}The map
\begin{eqnarray*}
\quotient{\cro{\ynorm}}{\fdiff[\ww C^{3},0,\tx{Id}]} & \longrightarrow & \Lambda_{\lambda}\left(\ynorm\right)\times\Lambda_{-\lambda}\left(\ynorm\right)\\
Y & \longmapsto & \left(\Phi_{\lambda},\Phi_{-\lambda}\right)\,\,
\end{eqnarray*}
is well-defined and bijective.
\end{thm}
In particular, the result states that Stokes diffeomorphisms only
depend on the class of $Y\in\cro{\ynorm}$ in the quotient ${\displaystyle \quotient{\cro{\ynorm}}{\fdiff[\ww C^{3},0,\tx{Id}]}}$.
We will give a description of this set of invariants in terms of power
series in the \emph{space of leaves} in section \ref{sec:Sectorial-isotropies-and}.
\begin{rem}
\label{rem:New difficulties}In \cite{bittmann2} we have proved a
theorem of sectorial normalizing map analogous to the classical one
due to Hukuhara-Kimura-Matuda for saddle-nodes in $\left(\ww C^{2},0\right)$
\cite{HKM}, generalized later by Stolovitch in any dimension in \cite{Stolo}.
Unlike the method based on a fixed point theorem used by these authors,
we have used a more geometric approach (following the works of Teyssier~\cite{Teyssier03,teyssier2004equation})
based on the resolution of an homological equation, by integrating
a well chosen 1-form along asymptotic paths. This latter approach
turned out to be more efficient to deal with the fact that $Y_{\mid\acc{x=0}}$
is not necessarily linearizable. Indeed, if we try to adapt the proof
of \cite{Stolo}, one of the first new main difficulties is that in
the irregular systems that needs to be solved by a fixed point method
(for instance equation $\left(2.7\right)$ in the cited paper), the
non-linear terms would not be divisible by the independent variable
(\emph{i.e.} the time) in our situation. This would complicate the
different estimates that are done later in the cited work. This was
the first main new phenomena we have met. 

In contrast to a result of \cite{bittmann2} which states that the
only sectorial isotropy (tangent to the identity) of the normal form
over wide sectors (of opening $>\pi$) is the identity, we will see
here that the situation is rather different over sector with narrow
opening. In order to prove both Theorems \ref{Th: Th drsn-1} and
\ref{thm: espace de module}, we will show that the Stokes diffeomorphisms
$\Phi_{\lambda}$ and $\Phi_{-\lambda}$ obtained from the germs of
sectorial normalizing maps $\Phi_{+}$ and $\Phi_{-}$,admit the identity
as Gevrey-1 asymptotic expansion. In the cited reference we were only
able to establish that fact with a weaker notion of Gevrey-1 expansion.
The main difficulty is to prove that such a sectorial isotropy of
$\ynorm$ over the ``narrow'' sectors described above is necessarily
exponentially close to the identity (see Proposition \ref{prop: isotropies plates}).
This will be done \emph{via} a detailed analysis of these maps in
the space of leaves (see Definition \ref{def: space of leaves}).
In fact, this is the second main new difficulties we have met, which
is due to the presence of the ``resonant'' term 
\[
\frac{c_{m}\left(y_{1}y_{2}\right)^{m}\log\left(x\right)}{x}
\]
in the exponential expression of the first integrals of the vector
field (see $\left(\mbox{\ref{eq: integrales premieres}}\right)$).
In \cite{Stolo}, similar computations are done in subsection $3.4.1$.
In this part of the paper, infinitely many irregular differential
equations appear when identifying terms of same homogeneous degree.
The fact that $Y_{\mid\acc{x=0}}$ is linear implies that these differential
equations are all linear and independent of each others (\emph{i.e.
}they are not mixed together). In our situation, this is not the case
and then more complicated. Using a ``non-abelian'' version of the
Ramis-Sibuya theorem due to Martinet and Ramis \cite{MR82}, we prove
both sectorial normalizing maps $\Phi_{+}$ and $\Phi_{-}$ admit
the formal normalizing map $\hat{\Phi}$ as Gevrey-$1$ asymptotic
expansion in the corresponding sectorial domains. This establishes
the Gevrey-1 summability of $\hat{\Phi}$.
\end{rem}

\subsection{\label{sub: transversally symplectic}The transversally Hamiltonian
case }

~

In order to motivate the following definition, we refer to \cite{bittmann2}
where the study of the Painlevé case is performed.
\begin{defn}
\label{def: intro}Consider the rational 1-form 
\begin{eqnarray*}
\omega & := & \frac{\mbox{d}y_{1}\wedge\mbox{d}y_{2}}{x}~.
\end{eqnarray*}
We say that vector field $Y$ is \textbf{transversally Hamiltonian
}(with respect to $\omega$ and $\mbox{dx}$) if 
\begin{eqnarray*}
\mathcal{L}_{Y}\left(\tx dx\right)\in\left\langle \tx dx\right\rangle  & \mbox{ and } & \mathcal{L}_{Y}\left(\omega\right)\in\left\langle \tx dx\right\rangle \qquad.
\end{eqnarray*}
For any open sector $S\subset\ww C^{*}$, we say that a germ of sectorial
fibered diffeomorphism $\Phi$ in $S\times\left(\ww C^{2},0\right)$
is \textbf{transversally symplectic} (with respect to $\omega$ and
$\mbox{d}x$) if
\[
\Phi^{*}\left(\omega\right)\in\,\omega+\ps{\tx dx}\qquad
\]
 (here $\Phi^{*}\left(\omega\right)$ denotes the pull-back of $\omega$
by $\Phi$). 

We denote by $\sdiffid$ the group of transversally symplectic diffeomorphisms
which are tangent to the identity.\end{defn}
\begin{rem}
A fibered sectorial diffeomorphism $\Phi$ is transversally symplectic
if and only if $\det\left(\mbox{D}\Phi\right)=1$.\end{rem}
\begin{defn}
\label{def: nc dr th}A \textbf{transversally Hamiltonian doubly-resonant
saddle-node} is a transversally Hamiltonian vector field which is
conjugate, \emph{via }a (fibered) transversally symplectic diffeomorphism,
to one of the form 
\begin{eqnarray*}
Y & = & x^{2}\pp x+\Big(-\lambda y_{1}+F_{1}\left(x,\mathbf{y}\right)\Big)\pp{y_{1}}+\Big(\lambda y_{2}+F_{2}\left(x,\mathbf{y}\right)\Big)\pp{y_{2}}\,\,\,\,\,,
\end{eqnarray*}
with $\lambda\in\ww C^{*}$ and $F_{1},F_{2}$ analytic in $\left(\ww C^{3},0\right)$
and of order at least $2$.
\end{defn}
Notice that a transversally Hamiltonian doubly-resonant saddle-node
is necessarily strictly non-degenerate (since its residue is always
equal to $1$), and also div-integrable (see section 3).

We recall the second main result from \cite{bittmann1}. 
\begin{thm}
\cite{bittmann1}\label{thm: Th ham formel}~

Let $Y$ be a diagonal doubly-resonant saddle-node which is transversally
Hamiltonian. Then, there exists a unique formal fibered transversally
symplectic diffeomorphism $\hat{\Phi}$ tangent to identity such that:
\begin{eqnarray}
\hat{\Phi}_{*}\left(Y\right) & = & x^{2}\pp x+\left(-\lambda+a_{1}x-c\left(y_{1}y_{2}\right)\right)y_{1}\pp{y_{1}}+\left(\lambda+a_{2}x+c\left(y_{1}y_{2}\right)\right)y_{2}\pp{y_{2}}\nonumber \\
 & =: & \ynorm\,\,,\label{eq: fibered normal form-1-1-1}
\end{eqnarray}
where $\lambda\in\ww C^{*}$, $c\left(v\right)\in v\form v$ a formal
power series in $v=y_{1}y_{2}$ without constant term and $a_{1},a_{2}\in\ww C$
are such that $a_{1}+a_{2}=1$.
\end{thm}
The second main result in \cite{bittmann2} is the following.
\begin{thm}
\label{thm: Th ham}Let $Y$ be a transversally Hamiltonian doubly-resonant
saddle-node. Let $\hat{\Phi}$ be the unique formal normalizing map
given by Theorem \ref{thm: Th ham formel}. Then the associate sectorial
normalizing maps $\Phi_{+}$ and $\Phi_{-}$ given by Theorem \ref{Th: Th drsn}
are also transversally symplectic.
\end{thm}
Let us fix a normal form $\ynorm$ as in Theorem \ref{thm: Th ham},
and consider two sectorial domains ${\displaystyle S_{\lambda}\times\left(\ww C^{2},0\right)}$
and $S_{-\lambda}\times\left(\ww C^{2},0\right)$ as in (\ref{eq: petits secteurs}).
Then, the Stokes diffeomorphisms $\left(\Phi_{\lambda},\Phi_{-\lambda}\right)$
defined in the previous subsection as 
\[
\begin{cases}
\Phi_{\lambda}:=\left(\Phi_{+}\circ\Phi_{-}^{-1}\right)_{\mid S_{\lambda}\times\left(\ww C^{2},0\right)}\\
\Phi_{-\lambda}:=\left(\Phi_{-}\circ\Phi_{+}^{-1}\right)_{\mid S_{-\lambda}\times\left(\ww C^{2},0\right)} & ,
\end{cases}
\]
are also transversally symplectic. 
\begin{defn}
We denote by $\Lambda_{\lambda}^{\omega}\left(\ynorm\right)$ $\big($\emph{resp.
}$\Lambda_{-\lambda}^{\omega}\left(\ynorm\right)$$\big)$ the group
of germs of sectorial fibered isotropies of $\ynorm$, admitting the
identity as Gevrey-1 asymptotic expansion in sectorial domains of
the form ${\displaystyle S_{\lambda}\times\left(\ww C^{2},0\right)}$
${\displaystyle \left(resp.\, S_{-\lambda}\times\left(\ww C^{2},0\right)\right)}$,
and which are transversally symplectic. 
\end{defn}
Let us denote by $\cro{\ynorm}_{\omega}$ the set of germs of vector
fields which are formally conjugate to $\ynorm$\emph{ via }(formal)
transversally symplectic diffeomorphisms tangent to the identity.
As a consequence of Theorems (\ref{thm: espace de module}) and (\ref{thm: Th ham}),
we can now state our third main result:
\begin{thm}
\label{th: espace de module symplectic}The map 
\begin{eqnarray*}
\quotient{\cro{\ynorm}_{\omega}}{\sdiffid} & \longrightarrow & \Lambda_{\lambda}^{\omega}\left(\ynorm\right)\times\Lambda_{-\lambda}^{\omega}\left(\ynorm\right)\\
Y & \longmapsto & \left(\Phi_{\lambda},\Phi_{-\lambda}\right)\,\,
\end{eqnarray*}
is a well-defined bijection.
\end{thm}

\subsection{Outline of the paper}

~

In section $2$, we recall the different tools, notations and nomenclature
we will need regarding asymptotic expansion, Gevrey-1 series, 1-summability
and sectorial germs.

In section $3$, we prove the main theorems presented above, assuming
the Proposition \ref{prop: isotropies plates} holds.

In section $4$, we prove the key Proposition \ref{prop: isotropies plates}
by studying the automorphisms of the space of leaves. 

In section $5$, we give a description of the moduli space in Theorem
\ref{thm: espace de module} in terms of power series in the space
of leaves, and present some applications.

In section $6$, we present a generalization of Theorem \ref{thm: espace de module}
where we study the action of $\diff$ instead of $\fdiff$.

\tableofcontents{}

\section{Background}

We refer the reader to \cite{MR82}, \cite{malgrange1995sommation}
and \cite{ramis1993divergent} for a general and detailed introduction
to the theory of asymptotic expansion, Gevrey series and summability
(see also \cite{Stolo} for a useful discussion of these concepts).
We refer more precisely to \cite{bittmann2} when it comes to the
notion of \emph{weak 1-summability. \bigskip{}
}

We call $x\in\ww C$ the \emph{independent} variable and ${\displaystyle \mathbf{y}:=\left(y_{1},\dots,y_{n}\right)\in\ww C^{n}}$,
$n\in\ww N$, the \emph{dependent} variables. As usual we define ${\displaystyle \mathbf{y^{k}}:=y_{1}^{k_{1}}\dots y_{n}^{k_{n}}}$
for ${\displaystyle \mathbf{k}=\left(k_{1},\dots,k_{n}\right)\in\ww N^{n}}$,
and ${\displaystyle \abs{\mathbf{k}}=k_{1}+\dots+k_{n}}$. The notions
of asymptotic expansion, Gevrey series and 1-summability presented
here are always considered with respect to the independent variable
$x$ living in (open) sectors 
\[
S\left(r,\alpha,\beta\right)=\acc{x\in\ww C\mid0<\abs x<r\mbox{ and }\alpha<\arg\left(x\right)<\beta}\,\,,
\]
the dependent variable $\mathbf{y}$ belonging to poly-discs 
\[
\mathbf{D\left(0,r\right)}:=\acc{\mathbf{y}=\left(y_{1},\dots,y_{n}\right)\in\ww C^{n}\mid\abs{y_{1}}<r_{1},\dots\abs{y_{n}}<r_{n}}\,\,\,,
\]
of poly-radius ${\displaystyle \mathbf{r}=\left(r_{1},\dots,r_{n}\right)\in\left(\ww R_{>0}\right)^{n}}$.
Given an open subset $\cal U\subset\ww C^{n+1}$ , we denote by $\cal O\left({\cal U}\right)$
the algebra of holomorphic function in $\cal U$.

\subsection{Sectorial germs}

~

Let $\theta\in\ww R$, $\eta\in\ww R_{\geq0}$ and $n\in\ww N$. 
\begin{defn}
\label{def: sectorial germs}
\begin{enumerate}
\item An\emph{ x-sectorial neighborhood }(or simply \emph{sectorial neighborhood})
\emph{of the origin} (in $\ww C^{n+1}$) \emph{in the direction $\theta$
with opening $\eta$} is an open set $\cal S\subset\ww C^{n+1}$ such
that
\[
\cal S\supset S\left(r,\theta-\frac{\eta}{2}-\epsilon,\theta+\frac{\eta}{2}+\epsilon\right)\times\mathbf{D\left(0,r\right)}
\]
for some $r>0$, $\mathbf{r}\in\left(\ww R_{>0}\right)^{n}$ and $\epsilon>0$.
We denote by $\left(\germsect{\theta}{\eta},\leq\right)$ the directed
set formed by all such neighborhoods, equipped with the order relation
\begin{eqnarray*}
S_{1}\leq S_{2} & \Longleftrightarrow & S_{1}\supset S_{2}\,\,.
\end{eqnarray*}

\item The algebra of \emph{germs of holomorphic functions in a sectorial
neighborhood of the origin in the direction $\theta$ with opening
$\eta$} is the direct limit 
\[
\cal O\left(\cal S_{\theta,\eta}\right):=\underrightarrow{\lim}\,\cal O\left(\cal S\right)
\]
with respect to the directed system defined by $\acc{\cal O\left(\cal S\right):\cal S\in\germsect{\theta}{\eta}}$.
\end{enumerate}
\end{defn}
We now give the definition of \emph{a (germ of a) sectorial diffeomorphism}.
\begin{defn}
\label{def: sectorial diff}
\begin{enumerate}
\item Given an element $\cal S\in\cal S_{\theta,\eta}$, we denote by $\fdiff[\cal S,\tx{Id}]$
the set of holomorphic diffeomorphisms of the form 
\begin{eqnarray*}
\Phi:\cal S & \rightarrow & \Phi\left(\cal S\right)\\
\left(x,\mathbf{y}\right) & \mapsto & \left(x,\phi_{1}\left(x,\mathbf{y}\right),\phi_{2}\left(x,\mathbf{y}\right)\right)\,\,,
\end{eqnarray*}
such that ${\displaystyle \Phi\left(x,\mathbf{y}\right)-\tx{Id}\left(x,\mathbf{y}\right)=\tx O\left(\norm{x,\mathbf{y}}^{2}\right),\,\,\mbox{ as }\left(x,\mathbf{y}\right)\rightarrow\left(0,\mathbf{0}\right)\mbox{ in }\cal S.}$
\footnote{This condition implies in particular that $\Phi\left(\cal S\right)\in\germsect{\theta}{\eta}$.%
}
\item The set of \emph{germs of (fibered) sectorial diffeomorphisms in the
direction $\theta$ with opening $\eta$, tangent to the identity},
is the direct limit 
\[
\diffsect:=\underrightarrow{\lim}\,\fdiff[\cal S,\tx{Id}]
\]
with respect to the directed system defined by $\acc{\fdiff[\cal S,\tx{Id}]:\cal S\in\cal S_{\theta,\eta}}$.
We equip $\diffsect$ of a group structure as follows: given two germs
$\Phi,\Psi\in{\displaystyle \diffsect}$ we chose corresponding representatives
$\Phi_{0}\in\fdiff[\cal S,\tx{Id}]$ and $\Psi_{0}\in\fdiff[\cal T,\tx{Id}]$
with $\cal S,\cal T\in\cal S_{\theta,\eta}$ such that $\cal T\subset\Phi_{0}\left(\cal S\right)$
and let $\Psi\circ\Phi$ be the germ defined by $\Psi_{0}\circ\Phi_{0}$.
\footnote{One can prove that this definition is independent of the choice of
the representatives%
}
\end{enumerate}
\end{defn}
We will also need the notion of \emph{asymptotic sectors.}
\begin{defn}
\label{def: asymptotitc sector}An\emph{ (open) asymptotic sector
of the origin in the direction $\theta$ and with opening $\eta$}
is an open set $S\subset\ww C$ such that 
\[
S\in\bigcap_{0\leq\eta'<\eta}\cal S_{\theta,\eta'}\,\,.
\]
We denote by $\cal{AS}_{\theta,\eta}$ the set of all such (open)
asymptotic sectors. 
\end{defn}

\subsection{\label{sub: Strong-Gevrey-1-power} Gevrey-1 power series and 1-summability}

~

In this subsection we fix a formal power series 
\[
{\displaystyle \hat{f}\left(x,\mathbf{y}\right)=\sum_{k\geq0}f_{k}\left(\mathbf{y}\right)x^{k}=\sum_{\left(j_{0},\mathbf{j}\right)\in\ww N^{n+1}}f_{j_{0},\mathbf{j}}x^{j_{0}}\mathbf{y^{j}}\in\form{x,\mathbf{y}}}\,\,.
\]

\begin{defn}
\label{def:Gevrey-1_expansion}\label{prop: Borel_Laplace_summation}~
\begin{itemize}
\item An analytic (and bounded) function $f$ in a sectorial domain ${\displaystyle \sect r{\alpha}{\beta}\times\mathbf{D\left(0,r\right)}}$
admits $\hat{f}$ as \emph{Gevrey-1 asymptotic expansion} in this
domain, if for all closed sub-sector $S'\subset\sect r{\alpha}{\beta}$,
there exists $A,C>0$ such that: 
\[
\abs{f\left(x,\mathbf{y}\right)-\sum_{k=0}^{N-1}f_{k}\left(\mathbf{y}\right)x^{k}}\leq AC^{N}\left(N!\right)\abs x^{N}
\]
for all $N\in\ww N$ and $\left(x,\mathbf{y}\right)\in S'\times\mathbf{D}\left(\mathbf{0},\mathbf{r}\right)$.
\item A formal power series $\hat{f}\in\form{x,\mathbf{y}}$ is 1-summable
in the direction $\theta$ if and only if there exists a germ of a
sectorial holomorphic function $f_{\theta}\in\cal O\left(\cal S_{\theta,\pi}\right)$
which admits $\hat{f}$ as Gevrey-1 asymptotic expansion in some $\cal S\in\cal S_{\theta,\pi}$. 
\end{itemize}
\end{defn}
\begin{rem}
In the definition above, $f_{\theta}$ is unique $\big($as a germ
in ${\displaystyle \cal O\left(\cal S_{\theta,\pi}\right)}$$\big)$,
and is called the \emph{1-sum of $\hat{f}$ in the direction $\theta$}.\end{rem}
\begin{lem}
\label{lem:diff_alg_and_summability}The set $\Sigma_{\theta}\subset\form{x,\mathbf{y}}$
of $1$-summable power series in the direction $\theta$ is an algebra
closed under differentiation. Moreover the map
\begin{eqnarray*}
\Sigma_{\theta} & \longrightarrow & \cal O\left(\cal S_{\theta,\pi}\right)\\
\hat{f} & \longmapsto & f_{\theta}
\end{eqnarray*}
is an injective morphism of differential algebras.\end{lem}
\begin{prop}
\label{prop: compositon summable}Let $\hat{\Phi}\left(x,\mathbf{y}\right)\in\form{x,\mathbf{y}}$
be 1-summable in directions $\theta$ and $\theta-\pi$, and let $\Phi_{+}\left(x,\mathbf{y}\right)$
and $\Phi_{-}\left(x,\mathbf{y}\right)$ be its 1-sums directions
$\theta$ and $\theta-\pi$ respectively. Let also ${\displaystyle \hat{f}_{1}\left(x,\mathbf{z}\right),\dots,\hat{f}_{n}\left(x,\mathbf{z}\right)}$
be 1-summable in directions $\theta$, $\theta-\pi$, and $f_{1,+},\dots,f_{n,+}$,
and $f_{1,-},\dots,f_{n,-}$ be their 1-sums in directions $\theta$
and $\theta-\pi$ respectively. Assume that 
\begin{equation}
\hat{f}_{j}\left(0,\mathbf{0}\right)=0\mbox{, for all }j=1,\dots,n\,\,.\label{eq: condition composition sommable}
\end{equation}
 Then 
\[
\hat{\Psi}\left(x,\mathbf{z}\right):=\hat{\Phi}\left(x,\hat{f}_{1}\left(x,\mathbf{z}\right),\dots,\hat{f}_{n}\left(x,\mathbf{z}\right)\right)
\]
 is 1-summable in directions $\theta,\theta-\pi$, and its 1-sum in
the corresponding direction is 
\[
\Psi_{\pm}\left(x,\mathbf{z}\right):=\Phi_{\pm}\left(x,f_{1,\pm}\left(x,\mathbf{z}\right),\dots,f_{n,\pm}\left(x,\mathbf{z}\right)\right)\,\,\,,
\]
which is a germ of a sectorial holomorphic function in directions
$\theta$ and $\theta-\pi$.\end{prop}
\begin{proof}
See \cite{bittmann1}.
\end{proof}
Consider $\hat{Y}$ a formal singular vector field at the origin and
a formal fibered diffeomorphism $\hat{\varphi}:\left(x,\mathbf{y}\right)\mapsto\left(x,\hat{\phi}\left(x,\mathbf{y}\right)\right)$.
Assume that both $\hat{Y}$ and $\hat{\varphi}$ are 1-summable in
directions $\theta$ and $\theta-\pi$, for some $\theta\in\ww R$,
and denote by $Y_{+},Y_{-}$ (\emph{resp. $\varphi_{+},\varphi_{-}$})
their 1-sums in directions $\theta$ and $\theta-\pi$ respectively.
As a consequence of Proposition \ref{prop: compositon summable} and
Lemma \ref{lem:diff_alg_and_summability}, we can state the following:
\begin{cor}
\label{cor: summability push-forward}Under the assumptions above,
$\hat{\varphi}_{*}\left(\hat{Y}\right)$ is 1-summable in both directions
$\theta$ and $\theta-\pi$, and its 1-sums in these directions are
$\varphi_{+}\left(Y_{+}\right)$ and $\varphi_{-}\left(Y_{-}\right)$
respectively.
\end{cor}

\subsection{An important result by Martinet and Ramis}

~

We are going to make an essential use of an isomorphism theorem proved
in \cite{MR82}. This result is of paramount importance in the present
paper since it will be a key tool in the proofs of both Theorems~\ref{Th: Th drsn}
and \ref{thm: espace de module} (see section~\ref{sec: analytic classification}).

Let us consider two open asymptotic sectors $\cal S$ and $\cal S'$
at the origin in directions $\theta$ and $\theta-\pi$ respectively,
both of opening $\pi$: 
\begin{eqnarray*}
S & \in & \cal{AS}_{\theta,\pi}\\
S' & \in & \cal{AS}_{\theta-\pi,\pi}
\end{eqnarray*}
(see Definition \ref{def: asymptotitc sector}). In this particular
setting, the cited theorem can be stated as follows. 
\begin{thm}
\label{th: martinet ramis}\cite[Théorème 5.2.1]{MR82} Consider a
pair of germs of sectorial diffeomorphisms 
\[
\left(\varphi,\varphi'\right)\in\diffsect[\theta][0]\times\diffsect[\theta-\pi][0]
\]
such that $\varphi$ and $\varphi'$ extend analytically and admit
the identity as Gevrey-1 asymptotic expansion in $S\times\left(\ww C^{2},0\right)$
and $S'\times\left(\ww C^{2},0\right)$ respectively. Then, there
exists a pair $\left(\phi_{+},\phi_{-}\right)$ of germs of sectorial
fibered diffeomorphisms 
\[
\left(\phi_{+},\phi_{-}\right)\in\diffsect[\theta+\frac{\pi}{2}][\eta]\times\diffsect[\theta-\frac{\pi}{2}][\eta]
\]
with $\eta\in\left]\pi,2\pi\right[$, which extend analytically in
$S_{+}\times\left(\ww C^{2},0\right)$ and $S_{-}\times\left(\ww C^{2},0\right)$
respectively, for some ${\displaystyle S_{+}\in\cal{AS}_{\theta+\frac{\pi}{2},2\pi}}$
and ${\displaystyle S_{-}\in\cal{AS}_{\theta-\frac{\pi}{2},2\pi}}$
, such that:
\[
\begin{cases}
\phi_{+}\circ\left(\phi_{-}\right)_{\mid S\times\left(\ww C^{2},0\right)}^{-1}=\varphi\\
\phi_{+}\circ\left(\phi_{-}\right)_{\mid S'\times\left(\ww C^{2},0\right)}^{-1}=\varphi' & \,\,.
\end{cases}
\]
There also exists a formal diffeomorphism $\hat{\phi}$ which is tangent
to the identity, such that $\phi_{+}$ and $\phi_{-}$ both admit
$\hat{\phi}$ as Gevrey-1 asymptotic expansion in $S_{+}\times\left(\ww C^{2},0\right)$
and $S_{-}\times\left(\ww C^{2},0\right)$ respectively.
\end{thm}
In particular, in the theorem above $\hat{\phi}$ is $1-$summable
in every direction except $\theta$ and $\theta-\pi$, and its 1-sums
in directions $\theta+\frac{\pi}{2}$ and $\theta-\frac{\pi}{2}$
respectively are $\phi_{+}$ and $\phi_{-}$. For future use, we are
going to prove a \emph{``transversally symplectic'' }version of
the above theorem.
\begin{cor}
\label{cor: MR symplectic}With the assumptions and notations of Theorem
\ref{th: martinet ramis}, if $\varphi$ and $\varphi'$ both are
transversally symplectic (see Definition \ref{def: intro}), then
there exists a germ of an analytic fibered diffeomorphism $\psi\in\fdiff[\ww C^{3},0,\tx{Id}]$
(tangent to the identity), such that
\[
\sigma_{+}:=\phi_{+}\circ\psi\mbox{ and }\sigma_{-}:=\phi_{-}\circ\psi
\]
 both are transversally symplectic. Moreover we also have:
\[
\begin{cases}
\sigma_{+}\circ\left(\sigma_{-}\right)_{\mid S\times\left(\ww C^{2},0\right)}^{-1}=\varphi\\
\sigma_{+}\circ\left(\sigma_{-}\right)_{\mid S'\times\left(\ww C^{2},0\right)}^{-1}=\varphi' & \,\,.
\end{cases}
\]
\end{cor}
\begin{proof}
We recall that for any germ $\varphi$ of a sectorial fibered diffeomorphism
which is tangent to the identity, $\varphi$ is transversally symplectic
if and only if $\det\left(\tx D\varphi\right)=1$. 

First of all, let us show that 
\[
\det\left(\tx D\phi_{+}\right)=\det\left(\tx D\phi_{-}\right)\mbox{ in }\left(S_{+}\cap S_{-}\right)\times\left(\ww C^{2},0\right)\,\,\,.
\]
Since $\phi_{+}$ and $\phi_{-}$ both are sectorial fibered diffeomorphism
which are tangent to the identity and transversally symplectic, then
\[
\det\left(\phi_{+}\circ\left(\phi_{-}\right)_{\mid\left(S_{+}\cap S_{-}\right)\times\left(\ww C^{2},0\right)}^{-1}\right)=1\,\,.
\]
The \emph{chain rule} implies immediately that 
\[
\det\left(\tx D\phi_{+}\right)=\det\left(\tx D\phi_{-}\right)\mbox{ in }\left(S_{+}\cap S_{-}\right)\times\left(\ww C^{2},0\right)\,\,\,.
\]
Thus, this equality allows us to define (thanks to the Riemann's Theorem
of removable singularities) a germ of analytic function $f\in\cal O\left(\ww C^{3},0\right)$.
Notice that $f\left(0,0,0\right)=1$ because $\phi_{+}$ and $\phi_{-}$
are tangent to the identity. Now, let us look for an element $\psi\in\fdiff[\ww C^{3},0,\tx{Id}]$
of the form 
\begin{equation}
\psi:\left(x,y_{1},y_{2}\right)\mapsto\left(x,\psi_{1}\left(x,\mathbf{y}\right),y_{2}\right)\label{eq: rectififcation symplectic}
\end{equation}
 such that 
\[
\sigma_{+}:=\phi_{+}\circ\psi\mbox{ and }\sigma_{-}:=\phi_{-}\circ\psi
\]
 both be transversally symplectic. An easy computation gives: 
\[
\det\left(\sigma_{\pm}\right)=\left(\det\left(\tx D\phi_{\pm}\right)\circ\psi\right)\det\left(\tx D\psi\right)=1
\]
\emph{i.e.}
\[
\left(f\circ\psi\right)\det\left(\mbox{D}\psi\right)=1\,\,\,.
\]
According to (\ref{eq: rectififcation symplectic}), we must have:
\begin{equation}
\left(f\circ\psi\right)\ppp{\psi_{1}}{y_{1}}=1\,\,.\label{eq: equa diff rectification symplectic}
\end{equation}
Let us define 
\[
F\left(x,y_{1},y_{2}\right):=\int_{0}^{y_{1}}f\left(x,z,y_{2}\right)\tx dz\,\,,
\]
so that (\ref{eq: equa diff rectification symplectic}) can be integrated
as 
\[
F\circ\psi=y_{1}+h\left(x,y_{2}\right)\,\,\,,
\]
for some $h\in\germ{x,y_{2}}$. Notice that 
\[
\ppp F{y_{1}}\left(0,0,0\right)=1
\]
since $f\left(0,0,0\right)=1$. Let us chose $h=0$. Then, we have
to solve 
\[
F\circ\psi=y_{1}\,\,\,,
\]
with unknown $\psi\in\fdiff[\ww C^{3},0,\tx{Id}]$ as in (\ref{eq: equa diff rectification symplectic}).
If we define
\[
\Phi:\left(x,\mathbf{y}\right)\mapsto\left(x,F\left(x,\mathbf{y}\right),y_{2}\right)\,\,,
\]
the latter problem is equivalent to find $\psi$ as above such that:
\[
\Phi\circ\psi=\tx{Id}\,\,\,.
\]
Since $\tx D\Phi_{0}=\tx{Id}$, the inverse function theorem gives
us the existence of the germ $\psi=\Phi^{-1}\in\fdiff[\ww C^{3},0,\tx{Id}]$
.
\end{proof}

\subsection{1-summability implies weakly 1-summability}

~

Any function $f\left(x,\mathbf{y}\right)$ analytic in a domain $\cal U\times\mathbf{D}\left(\mathbf{0},\mathbf{r}\right)$,
where $\cal U\subset\ww C$ is open, and bounded in any domain $\cal U\times\overline{\mathbf{D}}\left(\mathbf{0},\mathbf{r'}\right)$
with $r'_{1}<r_{1},\dots,r'_{n}<r_{n}$, can be written 
\begin{equation}
f\left(x,\mathbf{y}\right)=\sum_{\mathbf{j}\in\ww N^{n}}F_{\mathbf{j}}\left(x\right)\mathbf{y}^{\mathbf{j}}\qquad,\label{eq: developpement selon y}
\end{equation}
where for all $\mathbf{j}\in\ww N^{n}$, $F_{\mathbf{j}}$ is analytic
and bounded on $\cal U$, and defined \emph{via }the Cauchy formula:
\[
F_{\mathbf{j}}\left(x\right)=\frac{1}{\left(2i\pi\right)^{n}}\int_{\abs{z_{1}}=r'_{1}}\dots\int_{\abs{z_{n}}=r'_{n}}\frac{f\left(x,\mathbf{z}\right)}{\left(z_{1}\right)^{j_{1}+1}\dots\left(z_{n}\right)^{j_{n}+1}}\mbox{d}z_{n}\dots\mbox{d}z_{1}\qquad.
\]
Notice that the convergence of the series above is uniform on every
compact with respect to $x$ and $\mathbf{y}$.

In the same way, any formal power series $\hat{f}\left(x,\mathbf{y}\right)\in\form{x,\mathbf{y}}$
can be written as 
\[
{\displaystyle \hat{f}\left(x,\mathbf{y}\right)=\sum_{\mathbf{j}\in\ww N^{n}}\hat{F}_{\mathbf{j}}}\left(x\right)\mathbf{y}^{\mathbf{j}}\,\,.
\]

We present here a weaker notion of 1-summability that we will also
need.
\begin{defn}
~
\begin{itemize}
\item A function 
\[
{\displaystyle f\left(x,\mathbf{y}\right)=\sum_{\mathbf{j}\in\ww N^{n}}F_{\mathbf{j}}\left(x\right)\mathbf{y^{j}}}
\]
analytic and bounded in a domain $\sect r{\alpha}{\beta}\times\mathbf{D}\left(\mathbf{0},\mathbf{r}\right)$,
admits $\hat{f}$ as \textbf{weak Gevrey-1 asymptotic expansion} in
$x\in\sect r{\alpha}{\beta}$, if for all $\mathbf{j}\in\ww N^{n}$,
$F_{\mathbf{j}}$ admits $\hat{F}_{\mathbf{j}}$ as Gevrey-1 asymptotic
expansion in $\sect r{\alpha}{\beta}$.
\item The formal power series $\hat{f}$ is said to be \textbf{weakly 1-summable
in the direction $\theta\in\ww R$}, if the following conditions hold:

\begin{itemize}
\item for all $\mathbf{j}\in\ww N^{n}$, $\hat{F}_{\mathbf{j}}\left(x\right)\in\form x$
is 1-summable in the direction $\theta$, whose 1-sum in the direction
$\theta$ is denoted by $F_{\mathbf{j}}$;
\item the series ${\displaystyle f_{\theta}\left(x,\mathbf{y}\right):=\sum_{\mathbf{j}\in\ww N^{n}}F_{\mathbf{j}}\left(x\right)\mathbf{y^{j}}}$
defines a germ of a sectorial holomorphic function in a domain of
the form 
\[
S\left(r,\theta-\frac{\pi}{2}-\epsilon,\theta+\frac{\pi}{2}+\epsilon\right)\times\mathbf{D\left(0,r\right)}\,\,.
\]

\end{itemize}

In this case, $f_{\theta}\left(x,\mathbf{y}\right)$ is called \textbf{the
weak 1-sum of $\hat{f}$ in the direction $\theta$}.

\end{itemize}
\end{defn}
The following proposition is an analogue of Proposition \ref{prop: compositon summable}
for weak 1-summable formal power series, with the a stronger condition
instead of (\ref{eq: condition composition sommable}).
\begin{prop}
\label{prop: composition faible}Let 
\[
{\displaystyle \hat{\Phi}\left(x,\mathbf{z}\right)=\sum_{\mathbf{j}\in\ww N^{n}}\hat{\Phi}_{\mathbf{j}}\left(x\right)\mathbf{z^{j}}\in\form{x,\mathbf{y}}}
\]
and 
\[
{\displaystyle \hat{f}^{\left(k\right)}\left(x,\mathbf{y}\right)=\sum_{\mathbf{j}\in\ww N^{n}}\hat{F}_{\mathbf{j}}^{\left(k\right)}\left(x\right)\mathbf{y^{j}}\in\form{x,\mathbf{y}}}\,\,,
\]
for $k=1,\dots,n$, be n+1 formal power series which are weakly 1-summable
in directions $\theta$ and $\theta-\pi$, and let us denote by $\Phi_{+},f_{+}^{\left(1\right)},\dots,f_{+}^{\left(n\right)}$
$\big($\emph{resp. }$\Phi_{-},f_{-}^{\left(1\right)},\dots,f_{-}^{\left(n\right)}$$\big)$
their respective weak 1-sums in the direction $\theta$ (\emph{resp.}
$\theta-\pi$). Assume that $\hat{F}_{\mathbf{0}}^{\left(k\right)}=0$
for all $k=1,\dots,n$. Then, 
\[
\hat{\Psi}\left(x,\mathbf{y}\right):=\hat{\Phi}\left(x,\hat{f}^{\left(1\right)}\left(x,\mathbf{y}\right),\dots,\hat{f}^{\left(n\right)}\left(x,\mathbf{y}\right)\right)
\]
 is weakly 1-summable directions $\theta$ and $\theta-\pi$, and
its 1-sum in the corresponding direction is 
\[
\Psi_{\pm}\left(x,\mathbf{y}\right)=\Phi_{\pm}\left(x,f_{\pm}^{\left(1\right)}\left(x,\mathbf{y}\right),\dots,f_{\pm}^{\left(n\right)}\left(x,\mathbf{y}\right)\right)\,\,\,,
\]
which is a germ of a sectorial holomorphic function in the direction
$\theta$ (\emph{resp.} $\theta-\pi$) with opening $\pi$.\end{prop}
\begin{proof}
See \cite{bittmann2}.
\end{proof}
As proved in \cite{bittmann2}, the next corollary gives the link
between 1-summability in some direction and weak 1-summability in
the same direction (we refer to \cite{bittmann2}, Definition $\tx{2.8}$,
or to \cite{MR82}, section $\tx{IV}$, for a definition of the norm
$\norm{\cdot}_{\lambda,\theta,\delta,\rho}$ associate to the space
of 1-summable formal power series in the direction $\theta$).
\begin{cor}
Let 
\begin{eqnarray*}
{\displaystyle \hat{f}\left(x,\mathbf{y}\right)} & = & \sum_{\mathbf{j}\in\ww N^{n}}\hat{F}_{\mathbf{j}}\left(x\right)\mathbf{y^{j}}\in\form{x,\mathbf{y}}
\end{eqnarray*}
 be a formal power series. Then, $\hat{f}$ is 1-summable in the direction
$\theta\in\ww R$, of 1-sum ${\displaystyle f\in\cal O\left(\cal S_{\theta,\pi}\right)}$,
if and only if the following two conditions hold:
\begin{itemize}
\item $\hat{f}$ is weakly 1-summable in the direction $\theta$, \emph{i.e.}
there exists $\lambda,\delta,\rho$ such that $\forall\mathbf{j}\in\ww N^{n}$,
${\displaystyle \norm{\hat{F}_{\mathbf{j}}}_{\lambda,\theta,\delta,\rho}<\infty}$
\item the power series ${\displaystyle \sum_{\mathbf{j}\in\ww N^{n}}\norm{\hat{F}_{\mathbf{j}}}_{\lambda,\theta,\delta,\rho}\mathbf{y}^{\mathbf{j}}}$\emph{
}\textup{\emph{is convergent in some polydisc $\mathbf{D\left(0,r\right)}$.}}
\end{itemize}
\end{cor}
\begin{proof}
See \cite{bittmann2}.
\end{proof}

\section{\label{sec: analytic classification}Proofs of the main theorems}

The aim of this section is to prove the main results of this paper,
assuming Proposition \ref{prop: isotropies plates} below holds.

\subsection{Analytic invariants: Stokes diffeomorphisms}

~

From now on, we fix a normal form 
\[
\ynorm=x^{2}\pp x+\left(-\lambda+a_{1}x-c\left(y_{1}y_{2}\right)\right)y_{1}\pp{y_{1}}+\left(\lambda+a_{2}x+c\left(y_{1}y_{2}\right)\right)y_{2}\pp{y_{2}}\,\,\,\,,
\]
with $\lambda\in\ww C^{*},$$\Re\left(a_{1}+a_{2}\right)>0$ and $c\in v\germ v$
vanishing at the origin. We denote by ${\displaystyle \cro{\ynorm}}$
the set of germs of holomorphic doubly-resonant saddle-nodes in $\left(\ww C^{3},0\right)$,
formally conjugate to $\ynorm$ by formal fibered diffeomorphisms
tangent to the identity. We refer the reader to Definition \ref{def: sectorial germs}
for notions relating to sectors.
\begin{defn}
~
\begin{itemize}
\item We define $\isotsect Y{\theta}{\eta}$, for all $\theta\in\ww R$
and $\eta\in\cro{0,2\pi}$, as the group of germs of sectorial fibered
isotropies of $\ynorm$ in sectorial domains in $\cal S_{\theta,\eta}$
(see Definition \ref{def: asymptotitc sector}), which are tangent
to the identity. 
\item We define $\Lambda_{\lambda}^{\left(\mbox{\ensuremath{\tx{weak}}}\right)}\left(\ynorm\right)$
$\left(\mbox{\emph{resp.} }\Lambda_{-\lambda}^{\left(\mbox{\ensuremath{\tx{weak}}}\right)}\left(\ynorm\right)\right)$
as the group of germs of sectorial fibered isotropies of $\ynorm$,
admitting the identity as \textbf{weak} Gevrey-1 asymptotic expansion
in sectorial domains of the form $S_{\lambda}\times\left(\ww C^{2},0\right)$
$\left(resp.\, S_{-\lambda}\times\left(\ww C^{2},0\right)\right)$,
where:
\begin{eqnarray*}
S_{\lambda} & \in & \cal{AS}_{\arg\left(\lambda\right),\pi}\\
S_{-\lambda} & \in & AS_{\arg\left(-\lambda\right),\pi}
\end{eqnarray*}
(see Definition \ref{def: asymptotitc sector}).
\end{itemize}
\end{defn}
We recall the notations given in the introduction: we have defined
$\Lambda_{\lambda}\left(\ynorm\right)$ $\left(\mbox{\emph{resp.} }\Lambda_{-\lambda}\left(\ynorm\right)\right)$
as the group of germs of sectorial fibered isotropies of $\ynorm$,
admitting the identity as Gevrey-1 asymptotic expansion in sectorial
domains of the form $S_{\lambda}\times\left(\ww C^{2},0\right)$ $\left(resp.\, S_{-\lambda}\times\left(\ww C^{2},0\right)\right)$.
It is clear that we have:
\[
\Lambda_{\pm\lambda}\left(\ynorm\right)\subset\Lambda_{\pm\lambda}^{\left(\mbox{\ensuremath{\tx{weak}}}\right)}\left(\ynorm\right)\subset\isotsect Y{\arg\left(\pm\lambda\right)}{\eta},\,\,\forall\eta\in\left]0,\pi\right[\,\,.
\]
 According to Theorem \ref{Th: Th drsn}, to any ${\displaystyle Y\in{\displaystyle \cro{\ynorm}}}$,
we can associate a pair of germs of sectorial fibered isotropies in
$S_{\lambda}\times\left(\ww C^{2},0\right)$ and $S_{-\lambda}\times\left(\ww C^{2},0\right)$
respectively, denoted by $\left(\Phi_{\lambda},\Phi_{-\lambda}\right)$:
\[
\begin{cases}
\Phi_{\lambda}:=\left(\Phi_{+}\circ\Phi_{-}^{-1}\right)_{\mid S_{\lambda}\times\left(\ww C^{2},0\right)}\in\isotsect Y{\arg\left(\lambda\right)}{\eta} & \,,\,\forall\eta\in\left]0,\pi\right[\\
\Phi_{-\lambda}:=\left(\Phi_{-}\circ\Phi_{+}^{-1}\right)_{\mid S_{-\lambda}\times\left(\ww C^{2},0\right)}\in\isotsect Y{\arg\left(-\lambda\right)}{\eta} & \,,\,\forall\eta\in\left]0,\pi\right[\,\,,
\end{cases}
\]
where $\left(\Phi_{+},\Phi_{-}\right)$ is the pair of the sectorial
normalizing maps given by Theorem \ref{Th: Th drsn}.
\begin{prop}
For any given $\eta\in\left]0,\pi\right[$ the map 
\begin{eqnarray*}
{\displaystyle {\displaystyle \cro{\ynorm}}} & \longrightarrow & \isotsect Y{\arg\left(\lambda\right)}{\eta}\times\isotsect Y{\arg\left(-\lambda\right)}{\eta}\\
Y & \longmapsto & \left(\Phi_{\lambda},\Phi_{-\lambda}\right)\,\,,
\end{eqnarray*}
actually ranges in $\Lambda_{\lambda}^{\left(\mbox{\ensuremath{\tx{weak}}}\right)}\left(\ynorm\right)\times\Lambda_{-\lambda}^{\left(\mbox{\ensuremath{\tx{weak}}}\right)}\left(\ynorm\right)$.\end{prop}
\begin{proof}
The fact that the sectorial normalizing maps $\Phi_{+},\Phi_{-}$
given by Theorem \ref{Th: Th drsn} both conjugate ${\displaystyle Y\in{\displaystyle \cro{\ynorm}}}$
to $\ynorm$ in the corresponding sectorial domains proves that the
arrow above is well-defined, with values in $\isotsect Y{\arg\left(\lambda\right)}{\eta}\times\isotsect Y{\arg\left(-\lambda\right)}{\eta}$,
for all $\eta\in\left]0,\pi\right[$. The fact that $\Phi_{\pm\lambda}$
admits the identity as weak Gevrey-1 asymptotic expansion in $S_{\pm\lambda}\times\left(\ww C^{2},0\right)$
comes from the fact that $\Phi_{+}$ and $\Phi_{-}$ admits the same
weak Gevrey-1 asymptotic expansion in $S_{\lambda}\times\left(\ww C^{2},0\right)$
and $S_{-\lambda}\times\left(\ww C^{2},0\right)$, and from Proposition
\ref{prop: composition faible}.
\end{proof}
The subgroup $\fdiff[\ww C^{3},0,\tx{Id}]\subset\fdiff$ formed by
fibered diffeomorphisms tangent to the identity acts naturally on
${\displaystyle {\displaystyle {\displaystyle \cro{\ynorm}}}}$ by
conjugacy. Now we show that the uniqueness of germs of sectorial normalizing
maps $\left(\Phi_{+},\Phi_{-}\right)$ implies that the Stokes diffeomorphisms
$\left(\Phi_{\lambda},\Phi_{-\lambda}\right)$ of a vector field ${\displaystyle Y\in{\displaystyle {\displaystyle \cro{\ynorm}}}}$
is invariant under the action of $\fdiff[\ww C^{3},0,\tx{Id}]$. Furthermore,
this map is one-to-one.
\begin{prop}
\label{prop: espace de module injection}The map 
\begin{eqnarray*}
{\displaystyle {\displaystyle {\displaystyle \cro{\ynorm}}}} & \longrightarrow & \Lambda_{\lambda}^{\left(\mbox{\ensuremath{\tx{weak}}}\right)}\left(\ynorm\right)\times\Lambda_{-\lambda}^{\left(\mbox{\ensuremath{\tx{weak}}}\right)}\left(\ynorm\right)\\
Y & \longmapsto & \left(\Phi_{\lambda},\Phi_{-\lambda}\right)
\end{eqnarray*}
 factorizes through a one-to-one map
\begin{eqnarray*}
\quotient{{\displaystyle {\displaystyle \cro{\ynorm}}}}{\fdiff[\ww C^{3},0,\tx{Id}]} & \longrightarrow & \Lambda_{\lambda}^{\left(\mbox{\ensuremath{\tx{weak}}}\right)}\left(\ynorm\right)\times\Lambda_{-\lambda}^{\left(\mbox{\ensuremath{\tx{weak}}}\right)}\left(\ynorm\right)\\
Y & \longmapsto & \left(\Phi_{\lambda},\Phi_{-\lambda}\right)\,\,.
\end{eqnarray*}
\end{prop}
\begin{rem}
This very result means that the Stokes diffeomorphisms encode completely
the class of $Y$ in the quotient $\quotient{\left[\ynorm\right]}{\fdiff[\ww C^{3},0,\tx{Id}]}$
as they separate conjugacy classes.\end{rem}
\begin{proof}
First of all, let us prove that the latter map is well-defined. Let
${\displaystyle Y,\tilde{Y}\in{\displaystyle {\displaystyle \cro{\ynorm}}}}$
and $\Theta\in\fdiff[\ww C^{3},0,\tx{Id}]$ be such that $\Theta_{*}\left(Y\right)=\tilde{Y}$.
We denote by $\Phi_{\pm}$ (\emph{resp.} $\tilde{\Phi}{}_{\pm}$)
the sectorial normalizing maps of $Y$ (\emph{resp. $\tilde{Y}$}),
and $\left(\Phi_{\lambda},\Phi_{-\lambda}\right)$ $\Big($\emph{resp.}
$\left(\tilde{\Phi}_{\lambda},\tilde{\Phi}_{-\lambda}\right)$$\Big)$
the Stokes diffeomorphisms of $Y$ (\emph{resp.} $\tilde{Y}$). By
assumption, $\tilde{\Phi}_{\pm}\circ\Theta$ is also a germ of a sectorial
fibered normalization of $Y$ in $S_{\pm}\times\left(\ww C^{2},0\right)$,
which is tangent to the identity. Thus, according to the uniqueness
statement in Theorem \ref{Th: Th drsn}: 
\[
\Phi_{\pm}=\tilde{\Phi}_{\pm}\circ\Theta\,\,.
\]
Consequently, in $S_{\pm\lambda}\times\left(\ww C^{2},0\right)$ we
have 
\begin{eqnarray*}
\Phi_{\lambda} & = & \left(\Phi_{+}\circ\Phi_{-}^{-1}\right)_{\mid S_{\lambda}\times\left(\ww C^{2},0\right)}\\
 & = & \tilde{\Phi}_{+}\circ\Theta\circ\Theta^{-1}\circ\tilde{\Phi}_{-}\\
 & = & \tilde{\Phi}_{\lambda}\,\,,
\end{eqnarray*}
and similarly
\begin{eqnarray*}
\Phi_{-\lambda} & = & \left(\Phi_{-}\circ\Phi_{+}^{-1}\right)_{\mid S_{-\lambda}\times\left(\ww C^{2},0\right)}\\
 & = & \tilde{\Phi}_{-}\circ\Theta\circ\Theta^{-1}\circ\left(\tilde{\Phi}_{+}\right)^{-1}\\
 & = & \tilde{\Phi}_{-\lambda}\,\,.
\end{eqnarray*}

Let us prove that the map is one-to-one. Let $Y,\tilde{Y}\in\cro{\ynorm}$
share the same Stokes diffeomorphisms $\left(\Phi_{\lambda},\Phi_{-\lambda}\right)$.
We denote by $\Phi_{\pm}$ (\emph{resp.} $\tilde{\Phi}{}_{\pm}$)
the germ of a sectorial fibered normalizing map of $Y$ (\emph{resp.
$\tilde{Y}$}) $S_{\pm}\times\left(\ww C^{2},0\right)$. We have:
\[
{\displaystyle \begin{cases}
\Phi_{+}\circ\left(\Phi_{-}\right)^{-1}=\Phi_{\lambda}=\tilde{\Phi}_{+}\circ\left(\tilde{\Phi}_{-}\right)^{-1} & \mbox{ in }S_{\lambda}\times\left(\ww C^{2},0\right)\\
\Phi_{-}\circ\left(\Phi_{+}\right)^{-1}=\Phi_{-\lambda}=\tilde{\Phi}_{-}\circ\left(\tilde{\Phi}_{+}\right)^{-1} & \mbox{ in }S_{-\lambda}\times\left(\ww C^{2},0\right)\,.
\end{cases}}
\]
Thus: 
\[
{\displaystyle \begin{cases}
\left(\tilde{\Phi}_{+}\right)^{-1}\text{\ensuremath{\circ}}\,\Phi_{+}=\left(\tilde{\Phi}_{-}\right)^{-1}\text{\ensuremath{\circ}}\,\Phi_{-} & \mbox{ in }S_{\lambda}\times\left(\ww C^{2},0\right)\\
\left(\tilde{\Phi}_{+}\right)^{-1}\text{\ensuremath{\circ}}\,\Phi_{+}=\left(\tilde{\Phi}_{-}\right)^{-1}\text{\ensuremath{\circ}}\,\Phi_{-} & \mbox{ in }S_{-\lambda}\times\left(\ww C^{2},0\right)\,.
\end{cases}}
\]
We can then define a map $\varphi$ analytic in a domain of the form
$\left(D\left(0,r\right)\backslash\acc 0\right)\times\mathbf{D\left(0,r\right)}$
by setting:
\[
{\displaystyle \begin{cases}
\varphi_{\mid S_{+}}=\left(\tilde{\Phi}_{+}\right)^{-1}\text{\ensuremath{\circ}}\,\Phi_{+} & \mbox{ in }S_{+}\\
\varphi_{\mid S_{-}}=\left(\tilde{\Phi}_{-}\right)^{-1}\text{\ensuremath{\circ}}\,\Phi_{-} & \mbox{ in }S_{-}\,.
\end{cases}}
\]
This map is analytic and bounded in $\left(D\left(0,r\right)\backslash\acc 0\right)\times\mathbf{D\left(0,r\right)}$,
and the Riemann singularity theorem tells us that this map can be
analytically extended to the entire poly-disc $D\left(0,r\right)\times\mathbf{D\left(0,r\right)}$.
As a conclusion, $\varphi\in\fdiff[\ww C^{3},0,\tx{Id}]$, $\Phi_{\pm}=\tilde{\Phi}_{\pm}\circ\varphi$
and $\varphi_{*}\left(Y\right)=\tilde{Y}$.
\end{proof}

\subsection{Proof of Theorem \ref{Th: Th drsn-1}: 1-summability of the formal
normalization}

~

We fix a normal form 
\[
\ynorm=x^{2}\pp x+\left(-\lambda+a_{1}x-c\left(y_{1}y_{2}\right)\right)y_{1}\pp{y_{1}}+\left(\lambda+a_{2}x+c\left(y_{1}y_{2}\right)\right)y_{2}\pp{y_{2}}\,\,\,\,,
\]
with $\lambda\in\ww C^{*},$$\Re\left(a_{1}+a_{2}\right)>0$ and $c\in v\germ v$
vanishing at the origin. In section \ref{sec:Sectorial-isotropies-and}
we will prove the following result.
\begin{prop}
\label{prop: isotropies plates}Any $\psi\in\Lambda_{\pm\lambda}^{\left(\tx{weak}\right)}\left(\ynorm\right)$
admits the identity as Gevrey-1 asymptotic expansion in $S_{\pm\lambda}\times\left(\ww C^{2},0\right)$.
In other words:
\[
\Lambda_{\pm\lambda}^{\left(\tx{weak}\right)}\left(\ynorm\right)=\Lambda_{\pm\lambda}\left(\ynorm\right)\,\,.
\]

\end{prop}
As a first consequence of Proposition \ref{prop: isotropies plates},
we obtain Theorem \ref{Th: Th drsn-1} which states that the formal
normalizing map from \cite{bittmann1} is in fact 1-summable.
\begin{proof}[Proof of Theorem \ref{Th: Th drsn-1}.]
\emph{}

Let us consider the unique germs of a sectorial normalizing map $\Phi_{+}$
and $\Phi_{-}$ in $S_{+}\times\left(\ww C^{2},0\right)$ and $S_{-}\times\left(\ww C^{2},0\right)$
respectively, and their associated Stokes diffeomorphisms:
\[
\begin{cases}
\Phi_{\lambda}=\left(\Phi_{+}\circ\Phi_{-}^{-1}\right)_{\mid S_{\lambda}\times\left(\ww C^{2},0\right)}\in\Lambda_{\lambda}^{\left(\tx{weak}\right)}\left(\ynorm\right)\\
\Phi_{-\lambda}=\left(\Phi_{-}\circ\Phi_{+}^{-1}\right)_{\mid S_{-\lambda}\times\left(\ww C^{2},0\right)}\in\Lambda_{-\lambda}^{\left(\tx{weak}\right)}\left(\ynorm\right) & \,\,\,.
\end{cases}
\]
According to Proposition \ref{prop: isotropies plates}, 
\[
\Lambda_{\pm\lambda}^{\left(\tx{weak}\right)}\left(\ynorm\right)=\Lambda_{\pm\lambda}\left(\ynorm\right)\,\,,
\]
so that $\Phi_{\lambda}$ and $\Phi_{-\lambda}$ both admit the identity
as Gevrey-1 asymptotic expansion, in $S_{\lambda}\times\left(\ww C^{2},0\right)$
and $S_{-\lambda}\times\left(\ww C^{2},0\right)$ respectively. Then,
Theorem \ref{th: martinet ramis} gives the existence of 
\[
\left(\phi_{+},\phi_{-}\right)\in\diffsect[\arg\left(i\lambda\right)][\eta]\times\diffsect[\arg\left(-i\lambda\right)][\eta]
\]
 for all $\eta\in\left]\pi,2\pi\right[$, such that:{\footnotesize{}
\[
\begin{cases}
\phi_{+}\circ\left(\phi_{-}\right)_{\mid S_{\lambda}\times\left(\ww C^{2},0\right)}^{-1}=\Phi_{\lambda}\\
\phi_{-}\circ\left(\phi_{+}\right)_{\mid S_{-\lambda}\times\left(\ww C^{2},0\right)}^{-1}=\Phi_{-\lambda} & \,,
\end{cases}
\]
}and the existence of a formal diffeomorphism $\hat{\phi}$ which
is tangent to the identity, such that $\phi_{+}$ and $\phi_{-}$
both admit $\hat{\phi}$ as Gevrey-1 asymptotic expansion in $S_{+}\times\left(\ww C^{2},0\right)$
and $S_{-}\times\left(\ww C^{2},0\right)$ respectively. In particular,
we have:
\[
\left(\left(\Phi_{+}\right)^{-1}\circ\phi_{+}\right)_{\substack{\mid\left(S_{\lambda}\cup S_{-\lambda}\right)\times\left(\ww C^{2},0\right)}
}=\left(\left(\Phi_{-}\right)^{-1}\circ\phi_{-}\right)_{\substack{\mid\left(S_{\lambda}\cup S_{-\lambda}\right)\times\left(\ww C^{2},0\right)}
}\,\,.
\]
This proves that the function $\Phi$ defined by $\left(\Phi_{+}\right)^{-1}\circ\phi_{+}$
in $S_{+}\times\left(\ww C^{2},0\right)$ and by $\left(\Phi_{-}\right)^{-1}\circ\phi_{-}$
in $S_{-}\times\left(\ww C^{2},0\right)$ is well-defined and analytic
in $D\left(0,r\right)\backslash\acc 0\times\mathbf{D\left(0,r\right)}$.
Since it is also bounded, it can be extended to an analytic map $\Phi$
in $D\left(0,r\right)\times\mathbf{D\left(0,r\right)}$ by Riemann's
theorem. Hence: 
\[
\begin{cases}
\phi_{+}=\Phi_{+}\circ\Phi\\
\phi_{-}=\Phi_{-}\circ\Phi & \,\,.
\end{cases}
\]
In particular, by composition, $\Phi_{+}$ and $\Phi_{-}$ both admit
$\hat{\phi}\circ\Phi^{-1}$ as Gevrey-1 asymptotic expansion in $S_{\lambda}\times\left(\ww C^{2},0\right)$
and $S_{-\lambda}\times\left(\ww C^{2},0\right)$ respectively. Since
$\Phi_{+}$ and $\Phi_{-}$ conjugates $Y$ to $\ynorm$ and since
the notion of asymptotic expansion commutes with the partial derivative
operators, the formal diffeomorphism $\hat{\phi}\circ\Phi^{-1}$ formally
conjugates $Y$ to $\ynorm$. Finally, notice that $\hat{\phi}\circ\Phi^{-1}$
is necessarily tangent to the identity. Hence, by uniqueness of the
formal normalizing map given by Theorem \ref{thm: forme normalel formelle},
we deduce that $\hat{\phi}\circ\Phi^{-1}=\hat{\Phi}$, the unique
formal normalizing map tangent to the identity. 
\end{proof}

\subsection{Proofs of Theorems \ref{thm: espace de module} and \ref{th: espace de module symplectic}}

~

Let us now present the proofs of Theorems \ref{thm: espace de module}
and \ref{th: espace de module symplectic}, assuming Proposition \ref{prop: isotropies plates}.

\subsubsection{Proof of Theorem \ref{thm: espace de module}}

~
\begin{proof}[Proof of Theorem \ref{thm: espace de module}.]
\emph{}
\end{proof}
Propositions \ref{prop: espace de module injection}, together with
Proposition \ref{prop: isotropies plates}, tell us that the considered
map is well-defined and one-to-one. It remains to prove that this
map is onto. Let 
\[
\begin{cases}
\Phi_{\lambda}\in\Lambda_{\lambda}\left(\ynorm\right)\\
\Phi_{-\lambda}\in\Lambda_{-\lambda}\left(\ynorm\right) & \,\,\,.
\end{cases}
\]
According to Theorem \ref{th: martinet ramis}, there exists 
\[
\left(\phi_{+},\phi_{-}\right)\in\diffsect[\arg\left(i\lambda\right)][\eta]\times\diffsect[\arg\left(-i\lambda\right)][\eta]
\]
with $\eta\in\left]\pi,2\pi\right[$, which extend analytically to
$S_{+}\times\left(\ww C^{2},0\right)$ and $S_{-}\times\left(\ww C^{2},0\right)$
respectively, such that:
\[
\phi_{\pm}\circ\left(\phi_{\mp}\right)_{\mid S_{\pm\lambda}\times\left(\ww C^{2},0\right)}^{-1}=\Phi_{\pm\lambda}
\]
and there also exists a formal diffeomorphism $\hat{\phi}$ which
is tangent to the identity, such that $\phi_{\pm}$ both admit $\hat{\phi}$
as asymptotic expansion in $S_{\pm}\times\left(\ww C^{2},0\right)$.
Let us consider the two germs of sectorial vector fields obtained
as
\[
Y_{\pm}:=\left(\phi_{\pm}^{-1}\right)_{*}\left(\ynorm\right)
\]
In particular, since $\hat{\phi}$ is the Gevrey-1 asymptotic expansion
of $\phi_{\pm}$, the vector fields $Y_{\pm}$ both admit $\left(\hat{\phi}\right)_{*}\left(\ynorm\right)$
as Gevrey-1 asymptotic expansion. The fact that $\phi_{+}\circ\left(\phi_{-}\right)^{-1}$
is an isotropy of $\ynorm$ implies immediately that $Y_{+}=Y_{-}$
on 
\begin{eqnarray*}
\left(S_{+}\cap S_{-}\right)\times\left(\ww C^{2},0\right) & = & \left(S_{\lambda}\cup S_{-\lambda}\right)\times\left(\ww C^{2},0\right)\,\,.
\end{eqnarray*}
Then, the vector field $Y$, which coincides with $Y_{\pm}$ in $S_{\pm}\times\left(\ww C^{2},0\right)$,
defines a germ of analytic vector field in $\left(\ww C^{3},0\right)$
by Riemann's theorem. By construction, ${\displaystyle Y\in\fdiff[\ww C^{3},0,\tx{Id}]_{*}\left(\ynorm\right)}$
and admits $\left(\Phi_{\lambda},\Phi_{-\lambda}\right)$ as Stokes
diffeomorphisms.

\subsubsection{Proof of Theorem \ref{th: espace de module symplectic}}

~

In a similar way, we prove now Theorem \ref{th: espace de module symplectic}.
\begin{proof}[Proof of Theorem \ref{th: espace de module symplectic}.]
\emph{}

Let $\ynorm\in\snodiag$ be a normal form which is also transversally
symplectic. We refer to subsection \ref{sub: transversally symplectic}
for the notations. It is clear from Theorems \ref{thm: espace de module}
and \ref{thm: Th ham} that the mapping is well-defined and one-to-one.
It remains to prove that it is also onto. Let 
\[
\begin{cases}
\Phi_{\lambda}\in\Lambda_{\lambda}^{\omega}\left(\ynorm\right)\\
\Phi_{-\lambda}\in\Lambda_{-\lambda}^{\omega}\left(\ynorm\right) & \,\,\,.
\end{cases}
\]
Since $\Lambda_{\lambda}^{\omega}\left(\ynorm\right)\subset\Lambda_{\lambda}\left(\ynorm\right)$
and $\Lambda_{-\lambda}^{\omega}\left(\ynorm\right)\subset\Lambda_{-\lambda}\left(\ynorm\right)$,
according to Theorem \ref{th: martinet ramis} there exists 
\[
\left(\phi_{+},\phi_{-}\right)\in\diffsect[\arg\left(i\lambda\right)][\eta]\times\diffsect[\arg\left(-i\lambda\right)][\eta]
\]
with $\eta\in\left]\pi,2\pi\right[$, which extend analytically in
$S_{+}\times\left(\ww C^{2},0\right)$ and $S_{-}\times\left(\ww C^{2},0\right)$
respectively, such that:
\[
\phi_{\pm}\circ\left(\phi_{\mp}\right)_{\mid S_{\pm\lambda}\times\left(\ww C^{2},0\right)}^{-1}=\Phi_{\pm\lambda}
\]
and there also exists a formal diffeomorphism $\hat{\phi}$ which
is tangent to the identity, such that $\phi_{\pm}$ both admit $\hat{\phi}$
as Gevrey-1 asymptotic expansion in $S_{\pm}\times\left(\ww C^{2},0\right)$.
According to Corollary \ref{cor: MR symplectic}, there exists a germ
of an analytic fibered diffeomorphism $\psi\in\fdiff[\ww C^{3},0,\tx{Id}]$
(tangent to the identity), such that
\[
\sigma_{\pm}:=\phi_{\pm}\circ\psi
\]
both are transversally symplectic. Then, we have: 
\[
\sigma_{\pm}\circ\left(\Psi_{\mp}\right)_{\mid S_{\pm\lambda}\times\left(\ww C^{2},0\right)}^{-1}=\Phi_{\pm\lambda}~.
\]
 The end of the proof goes exactly as at the end of the proof of the
previous theorem. 
\end{proof}

\section{\label{sec:Sectorial-isotropies-and}Sectorial isotropies and space
of leaves: proof of Proposition \ref{prop: isotropies plates}}

~

A normal form 
\[
\ynorm=x^{2}\pp x+\left(-\lambda+a_{1}x-c\left(y_{1}y_{2}\right)\right)y_{1}\pp{y_{1}}+\left(\lambda+a_{2}x+c\left(y_{1}y_{2}\right)\right)y_{2}\pp{y_{2}}
\]
is fixed for some $\lambda\in\ww C^{*},$ $\Re\left(a_{1}+a_{2}\right)>0$
and $c\in v\germ v$ (vanishing at the origin). The aim of this section
is to prove Proposition \ref{prop: isotropies plates} stated in Section
\ref{sec: analytic classification}.

Let us denote $a:=\res{\ynorm}=a_{1}+a_{2}$, $m:={\displaystyle \frac{1}{a}}$
and 
\[
c\left(v\right)=\sum_{k=1}^{+\infty}c_{k}v^{k}\,\,.
\]
If $m\notin\ww N_{>0}$, we set $c_{m}:=0$. We also define the following
power series 
\begin{eqnarray*}
\tilde{c}\left(v\right) & = & m\sum_{k\neq m}\frac{c_{k}}{k-m}v^{k}\,\,,
\end{eqnarray*}
and we notice that $\tilde{c}\left(v\right)\in v\germ v$.

\subsection{Sectorial first integrals and the space of leaves}

~

In a sectorial neighborhood of the origin of the form $S_{\lambda}\times\left(\ww C^{2},0\right)$
$\left(resp.\, S_{-\lambda}\times\left(\ww C^{2},0\right)\right)$,with
$S_{\pm\lambda}\in\germsect{\arg\left(\pm\lambda\right)}{\epsilon}$
and $\epsilon\in\left]0,\pi\right[$, we can give three first integrals
of $\ynorm$ which are analytic in the considered domain. Let us start
with the following proposition.
\begin{prop}
The following quantities are first integrals of $\ynorm$, analytic
in $S_{\pm\lambda}\times\left(\ww C^{2},0\right)$:

\begin{equation}
\begin{cases}
{\displaystyle w_{\pm\lambda}:=\frac{y_{1}y_{2}}{x^{a}}}\\
{\displaystyle h_{1,\pm\lambda}\left(x,\mathbf{y}\right):=y_{1}\exp\left(\frac{-\lambda}{x}+\frac{c_{m}\left(y_{1}y_{2}\right)^{m}\log\left(x\right)}{x}+\frac{\tilde{c}\left(y_{1}y_{2}\right)}{x}\right)x^{-a_{1}}}\\
{\displaystyle h_{2,\pm\lambda}\left(x,\mathbf{y}\right):=y_{2}\exp\left(\frac{\lambda}{x}-\frac{c_{m}\left(y_{1}y_{2}\right)^{m}\log\left(x\right)}{x}-\frac{\tilde{c}\left(y_{1}y_{2}\right)}{x}\right)x^{-a_{2}}}
\end{cases}\label{eq: integrales premieres}
\end{equation}
(we fix here a branch of the logarithm analytic in $S_{\pm\lambda}$,
and we write simply $h_{j}$ and $w$ instead of $h_{j,\pm\lambda}$
and $w_{\pm\lambda}$ respectively, if there is no ambiguity on the
sector $S_{\pm\lambda}$).

Moreover, we have the relation:
\[
h_{1}h_{2}=w\,\,.
\]
\end{prop}
\begin{proof}
It is an elementary computation.\end{proof}
\begin{rem}
In other words, in a sectorial domain, we can parametrize a leaf (which
is not in $\acc{x=0}$) of the foliation associated to $\ynorm$ by:
\begin{eqnarray}
 & \begin{cases}
{\displaystyle y_{1}\left(x\right)=h_{1}\exp\left(\frac{\lambda}{x}-c_{m}\left(h_{1}h_{2}\right)^{m}\log\left(x\right)-\frac{\tilde{c}\left(h_{1}h_{2}x^{a}\right)}{x}\right)x^{a_{1}}}\\
{\displaystyle y_{2}\left(x\right)=h_{2}\exp\left(-\frac{\lambda}{x}+c_{m}\left(h_{1}h_{2}\right)^{m}\log\left(x\right)+\frac{\tilde{c}\left(h_{1}h_{2}x^{a}\right)}{x}\right)x^{a_{2}}}
\end{cases}\label{eq: solutions parametrees}\\
 & \left(h_{1},h_{2}\right)\in\ww C^{2}\,.\nonumber 
\end{eqnarray}
\end{rem}
\begin{cor}
\label{cor: coordonn=0000E9es espace des feuiles}The map 
\begin{eqnarray*}
{\cal H_{\pm\lambda}}:S_{\pm\lambda}\times\left(\ww C^{2},0\right) & \rightarrow & S_{\pm\lambda}\times\ww C^{2}\\
\left(x,\mathbf{y}\right) & \mapsto & \left(x,h_{1,\pm\lambda}\left(x,\mathbf{y}\right),h_{2,\pm\lambda}\left(x,\mathbf{y}\right)\right)\,\,,
\end{eqnarray*}
(where $h_{1,\pm\lambda},h_{2,\pm\lambda}$ are defined in (\ref{eq: integrales premieres}))
is a sectorial germ of a fibered analytic map in $S_{\pm\lambda}\times\left(\ww C^{2},0\right)$,
which is into. Moreover, there exists an open neighborhood of the
origin in $\ww C^{2}$, denoted by $\ls{\pm}\subset\ww C^{2}$, such
that 
\[
{\cal H_{\pm\lambda}}\left(S_{\pm\lambda}\times\left(\ww C^{2},0\right)\right)=S_{\pm\lambda}\times\ls{\pm}\,\,.
\]
In particular, ${\cal H_{\pm}}$ induces a fibered biholomorphism
\[
S_{\pm\lambda}\times\left(\ww C^{2},0\right)\overset{{\cal H_{\pm\lambda}}}{\longrightarrow}S_{\pm\lambda}\times\ls{\pm}
\]
which conjugates $\ynorm$ to $x^{2}\pp x$, \emph{i.e.}
\[
\left({\cal H_{\pm\lambda}}\right)_{*}\left(\ynorm\right)=x^{2}\pp x\,\,.
\]
\end{cor}
\begin{defn}
\label{def: space of leaves}We call $\ls{\pm}$ \textbf{the space
of leaves of} $\ynorm$ in $S_{\pm\lambda}\times\left(\ww C^{2},0\right)$.\end{defn}
\begin{rem}
The set $\ls{\pm}$ depends on the choice of the neighborhood $\left(\ww C^{2},0\right)$,
but also on the choice of the sectorial neighborhood $S_{\pm\lambda}\in\germsect{\arg\left(\pm\lambda\right)}{\epsilon}$.
\end{rem}

\subsection{Sectorial isotropies in the space of leaves}

~

Now, we consider a germ of a sectorial isotropy $\psi_{\pm\lambda}\in\Lambda_{\pm\lambda}^{\left(\tx{weak}\right)}\left(\ynorm\right)$
and we denote by $\Gamma'_{\pm\lambda}$ the (germ of an) open subset
of $\ww C^{2}$ such that: 
\[
\cal H_{\pm\lambda}\circ\psi_{\pm}\left(S_{\pm\lambda}\times\left(\ww C^{2},0\right)\right)=S_{\pm\lambda}\times\Gamma'_{\pm\lambda}\,\,.
\]

\begin{prop}
\label{prop: isotropie espace des feuilles} With the notations and
assumptions above, the map 
\begin{eqnarray*}
\Psi_{\pm\lambda}:=\cal H_{\pm\lambda}\circ\psi_{\pm}\circ\cal H_{\pm\lambda}^{-1}:S_{\pm\lambda}\times\Gamma_{\pm\lambda} & \longrightarrow & S_{\pm\lambda}\times\Gamma'_{\pm\lambda}
\end{eqnarray*}
is a sectorial germ of a fibered biholomorphism from $S_{\pm\lambda}\times\Gamma_{\pm\lambda}$
to $S_{\pm\lambda}\times\Gamma'_{\pm\lambda}$, which is of the form:
\[
\Psi_{\pm\lambda}\left(x,h_{1},h_{2}\right)=\left(x,\Psi_{1,\pm\lambda}\left(h_{1},h_{2}\right),\Psi_{2,\pm\lambda}\left(h_{1},h_{2}\right)\right)\,\,.
\]
In particular, $\Psi_{1,\pm\lambda}$ and $\Psi_{2,\pm\lambda}$ are
analytic and depend only on $\left(h_{1},h_{2}\right)\in\Gamma_{\pm\lambda}$,
while $\Psi_{\pm\lambda}$ induces a biholomorphism (still written
$\Psi_{\pm\lambda}$): 
\begin{eqnarray*}
\Psi_{\pm\lambda}:\Gamma_{\pm\lambda} & \rightarrow & \Gamma'_{\pm\lambda}\\
\left(h_{1},h_{2}\right) & \mapsto & \left(\Psi_{1,\pm\lambda}\left(h_{1},h_{2}\right),\Psi_{2,\pm\lambda}\left(h_{1},h_{2}\right)\right)\,\,.
\end{eqnarray*}
\end{prop}
\begin{proof}
We only have to prove that $\Psi_{1,\pm\lambda}$ and $\Psi_{2,\pm\lambda}$
depend only on $\left(h_{1},h_{2}\right)\in\Gamma_{\pm\lambda}$.
By assumption, $\Psi_{\pm\lambda}$ is an isotropy of $x^{2}\pp x$:
\begin{eqnarray*}
\left(\Psi_{\pm\lambda}\right)_{*}\left(x^{2}\pp x\right) & = & x^{2}\pp x\,\,.
\end{eqnarray*}
We immediately obtain:
\[
\ppp{\Psi_{1,\pm\lambda}}x=\ppp{\Psi_{2,\pm\lambda}}x=0\,\,.
\]
 
\end{proof}
In the space of leaves $\Gamma_{\pm\lambda}$ equipped with coordinates
$\left(h_{1},h_{2}\right)$, we denote by $w$ the product of $h_{1}$
and $h_{2}$:
\[
w\left(h_{1},h_{2}\right):=h_{1}h_{2}\,\,.
\]
We define the two following quantities: 
\begin{equation}
\begin{cases}
f_{1}\left(x,w\right):=\exp\left(\frac{\lambda}{x}-c_{m}w^{m}\log\left(x\right)-\frac{\tilde{c}\left(wx^{a}\right)}{x}\right)x^{a_{1}}\\
f_{2}\left(x,w\right):=\exp\left(-\frac{\lambda}{x}+c_{m}w^{m}\log\left(x\right)+\frac{\tilde{c}\left(wx^{a}\right)}{x}\right)x^{a_{2}} & ,
\end{cases}\label{eq: f1 et f2}
\end{equation}
such that the leaves of the foliations are parametrized by: 
\[
\begin{cases}
y_{1}\left(x\right)=h_{1}f_{1}\left(x,h_{1}h_{2}\right)\\
y_{2}\left(x\right)=h_{2}f_{2}\left(x,h_{1}h_{2}\right)
\end{cases},\,\left(h_{1},h_{2}\right)\in\ww C^{2}\,.
\]
Notice that:
\begin{eqnarray*}
f_{1}\left(x,w\right)f_{2}\left(x,w\right) & = & x^{a}\,\,.
\end{eqnarray*}
Moreover, one checks immediately the following statement.
\begin{fact}
\label{fact: limites}For all $w\in\ww C$:
\[
\begin{cases}
\underset{\substack{x\rightarrow0\\
x\in S_{\lambda}
}
}{\lim}\abs{f_{1}\left(x,w\right)}=\underset{\substack{x\rightarrow0\\
x\in S_{-\lambda}
}
}{\lim}\abs{f_{2}\left(x,w\right)}=+\infty\\
\underset{\substack{x\rightarrow0\\
x\in S_{-\lambda}
}
}{\lim}\abs{f_{1}\left(x,w\right)}=\underset{\substack{x\rightarrow0\\
x\in S_{\lambda}
}
}{\lim}\abs{f_{2}\left(x,w\right)}=0 & \,\,.
\end{cases}
\]

\end{fact}
Using notations of Proposition \ref{prop: isotropie espace des feuilles},
we also assume from now on that $\left(\ww C^{2},0\right)=\mathbf{D\left(0,r\right)}$,
with $\mathbf{r}=\left(r_{1},r_{2}\right)\in\left(\ww R_{>0}\right)^{2}$
and $r_{1},r_{2}>0$ small enough so that 
\[
\psi_{\pm\lambda}\left(S_{\pm\lambda}\times\mathbf{D\left(0,r\right)}\right)\subset S_{\pm\lambda}\times\mathbf{D\left(0,r'\right)}
\]
for some $\mathbf{r'}=\left(r'_{1},r'_{2}\right)\in\left(\ww R_{>0}\right)^{2}$.
Let us now define in a general way the following set associated to
the sector $S_{\pm\lambda}$ and to a polydisc $\mathbf{D}\left(\mathbf{0},\tilde{\mathbf{r}}\right)$,
with $\tilde{\mathbf{r}}:=\left(\tilde{r}_{1},\tilde{r}_{2}\right)$.
\begin{defn}
\label{def: espace des feuilles =0000E0 rayon fixe}For all $x\in S_{\pm\lambda}$
et $\tilde{\mathbf{r}}:=\left(\tilde{r}_{1},\tilde{r}_{2}\right)\in\left(\ww R_{>0}\right)^{2}$,
we define
\[
\lsr{\pm}{x,\tilde{\mathbf{r}}}:=\acc{\left(h_{1},h_{2}\right)\in\ww C^{2}\mid\abs{h_{j}}\leq\frac{\tilde{r}_{j}}{\abs{f_{j}\left(x,h_{1}h_{2}\right)}}~,\,\mbox{for }j\in\left\{ 1,2\right\} }\,\,.
\]
We also consider the: 
\begin{eqnarray*}
\lsr{\pm}{\tilde{\mathbf{r}}} & := & \bigcup_{\substack{x\in S_{\pm\lambda}}
}\lsr{\pm}{x,\tilde{\mathbf{r}}}\\
 & = & \acc{\left(h_{1},h_{2}\right)\in\ww C^{2}\mid\exists x\in S_{\pm\lambda}\mbox{ s.t. }\abs{h_{j}}\leq\frac{\tilde{r}_{j}}{\abs{f_{j}\left(x,h_{1}h_{2}\right)}}~,~\mbox{for }j\in\left\{ 1,2\right\} }
\end{eqnarray*}
(\emph{cf. }figure \ref{fig:Illustration-de-l'espace}).
\end{defn}
Since we assume now that $\left(\ww C^{2},0\right)=\mathbf{D\left(0,r\right)}$,
then we have:
\[
\ls{\pm}=\lsr{\pm}{\mathbf{r}}\,\,,
\]
and 
\[
\lsp{\pm}\subset\lsr{\pm}{\mathbf{r}'}\,\,.
\]

\begin{figure}
\includegraphics[scale=0.22]{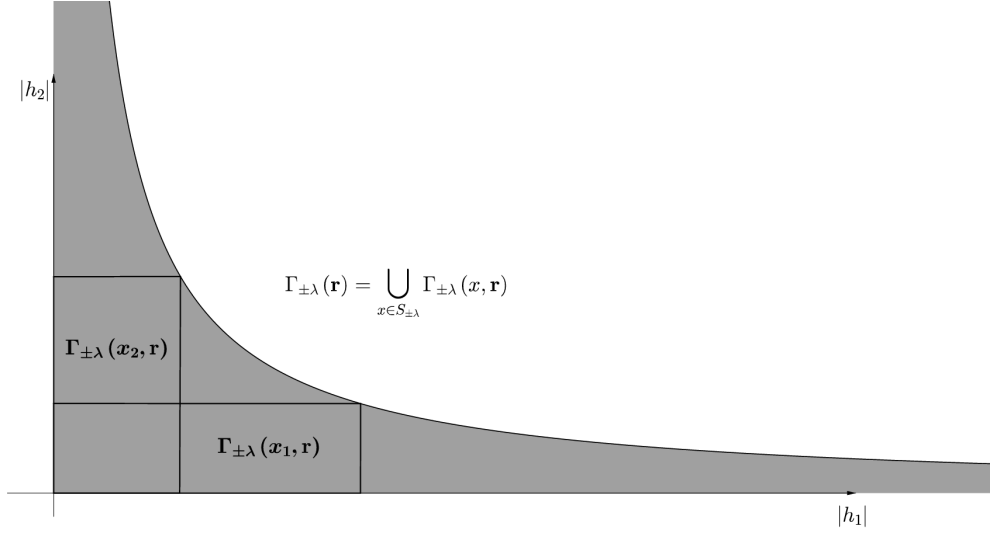}\protect\caption{\selectlanguage{french}%
\label{fig:Illustration-de-l'espace}\foreignlanguage{english}{Representation
of the space of leaves in terms of $\protect\abs{h_{1}}$ and $\protect\abs{h_{2}}$
when $c=0$: in this case, it is a\emph{ Reinhardt domain }(\emph{cf.
}\cite{hormander1973introduction})}.\selectlanguage{english}%
}
\end{figure}

\begin{rem}
~
\begin{enumerate}
\item It is important to notice that the particular form of $\Psi_{\pm\lambda}$
implies that the image of any fiber 
\[
\acc{x=x_{0}}\times\lsr{\pm}{x_{0},\mathbf{r}}
\]
by $\Psi_{\pm\lambda}$ is included in a fiber of the form 
\[
\acc{x=x_{0}}\times\lsr{\pm}{x_{0},\mathbf{r'}}\,\,.
\]

\item If $\left(h_{1},h_{2}\right)\in\lsr{\pm}{x,\mathbf{r}}$, then 
\begin{eqnarray*}
\abs{h_{1}h_{2}} & < & \frac{r_{1}r_{2}}{\abs{x^{a}}}\,\,.
\end{eqnarray*}

\item As $\left(h_{1},h_{2}\right)\in\Gamma_{\pm\lambda}$ varies the values
of $w=h_{1}h_{2}$ cover the whole $\ww C$.
\end{enumerate}
\end{rem}

\subsection{Action on the resonant monomial in the space of leaves}

~

Let us study the the action of $\Psi_{\pm\lambda}$ on the resonant
monomial $w=h_{1}h_{2}$ in the space of leaves.
\begin{lem}
\label{lem: isotropie monomone res esp des feuilles}We consider a
biholomorphism 
\begin{eqnarray*}
\Psi_{\pm\lambda}:\ls{\pm} & \tilde{\rightarrow} & \lsp{\pm}\\
\left(h_{1},h_{2}\right) & \mapsto & \left(\Psi_{1,\pm}\left(h_{1},h_{2}\right),\Psi_{2,\pm}\left(h_{1},h_{2}\right)\right)\,\,,
\end{eqnarray*}
such that for all $x\in S_{\pm\lambda}$, we have 
\[
\Psi_{\pm\lambda}\left(\lsr{\pm}{x_{0},\mathbf{r}}\right)\subset\lsr{\pm}{x_{0},\mathbf{r'}}\,\,.
\]
We also define $\Psi_{w,\pm\lambda}:=\Psi_{1,\pm\lambda}\Psi_{2,\pm\lambda}$.
Then, for all $n\in\ww N$, there exists entire (\emph{i.e. }analytic
over $\ww C$) functions $\Psi_{w,\lambda,n}$ and $\Psi_{w,-\lambda,n}$
such that 
\[
\begin{cases}
{\displaystyle \Psi_{w,\lambda}\left(h_{1},h_{2}\right)=\sum_{n\geq0}\Psi_{w,\lambda,n}\left(h_{1}h_{2}\right)h_{1}^{n}}\\
{\displaystyle \Psi_{w,-\lambda}\left(h_{1},h_{2}\right)=\sum_{n\geq0}\Psi_{w,-\lambda,n}\left(h_{1}h_{2}\right)h_{2}^{n}} & \,.
\end{cases}
\]
Moreover, the series above uniformly converge (for the $\sup$-norm)
in every subset of $\ls{\pm}$ of the form $\lsr{\pm}{\tilde{\mathbf{r}}}$,
with $\tilde{\mathbf{r}}:=\left(\tilde{r}_{1}\tilde{r}_{2}\right)$
and 
\[
0<\tilde{r}_{j}<r_{j}~~~,~j\in\left\{ 1,2\right\} 
\]
(\emph{cf.} Definition\emph{ }\ref{def: espace des feuilles =0000E0 rayon fixe}).
More precisely, for all $\tilde{r}_{1},\tilde{r}_{2},\delta>0$ such
that 
\[
0<\tilde{r}_{j}+\delta<r_{j}~~~,~j\in\left\{ 1,2\right\} 
\]
for all $x\in S_{\pm\lambda}$ and $w\in\ww C$ we have 
\begin{eqnarray*}
\abs{wx^{a}}\leq\tilde{r}_{1}\tilde{r}_{2} & \Longrightarrow & \begin{cases}
\abs{\Psi_{w,\lambda,n}\left(w\right)}\leq\frac{r'_{1}r'_{2}}{\abs{x^{a}}}\abs{\frac{f_{1}\left(x,w\right)}{\tilde{r}_{1}+\delta}}^{n}\\
\abs{\Psi_{w,-\lambda,n}\left(w\right)}\leq\frac{r'_{1}r'_{2}}{\abs{x^{a}}}\abs{\frac{f_{2}\left(x,w\right)}{\tilde{r}_{2}+\delta}}^{n}
\end{cases},\,\forall n\geq0\,.
\end{eqnarray*}
\end{lem}
\begin{proof}
Let us give the proof for $\Psi_{w,\lambda},\Psi_{1,\lambda}$ and
$\Psi_{2,\lambda}$ in $\Gamma_{\lambda}$ (the same proof applies
also for $\Psi_{w,-\lambda}$ in $\Gamma_{-\lambda}$ by exchanging
the role played by $h_{1}$ and $h_{2}$). We fix some $0<\tilde{r}_{j}<r_{j}$,
$j\in\left\{ 1,2\right\} $, and $\delta>0$ such that 
\[
0<\tilde{r}_{j}+\delta<r_{j}~~~,~j\in\left\{ 1,2\right\} .
\]
For a fixed value $w\in\ww C$, we consider the restriction of $\Psi_{w,\lambda}$
to the hypersurface $M_{w}:=\acc{h_{1}h_{2}=w}\cap\Gamma_{\lambda}$:
this restriction is analytic in $M_{w}$. The map
\[
\varphi_{w}:h_{1}\mapsto\Psi_{w,\lambda}\left(h_{1},\frac{w}{h_{1}}\right)
\]
 is analytic in 
\[
M_{w,1}:=\bigcup_{\substack{x\in S_{\lambda}\\
\abs{wx^{a}}<r_{1}r_{2}
}
}\Omega_{x,w}\,\,,
\]
where for all $x\in S_{\lambda}$ with $\abs{wx^{a}}<r_{1}r_{2}$,
the set $\Omega_{x,w}$ is the following annulus: 
\[
\Omega_{x,w}:=\acc{h_{1}\in\ww C\mid\abs{\frac{wf_{2}\left(x,w\right)}{r_{2}}}<\abs{h_{1}}<\abs{\frac{r_{1}}{f_{1}\left(x,w\right)}}}\,\,.
\]
In particular, $\varphi_{w}$ admits a Laurent expansion 
\[
\varphi_{w}\left(h_{1}\right)=\Psi_{w,+}\left(h_{1},\frac{w}{h_{1}}\right)=\sum_{n\geq-L}\Psi_{w,+,n}\left(w\right)h_{1}^{n}
\]
in every annulus $\Omega_{x,w}$, with $x\in S_{\lambda}$ such that
$\abs{wx^{a}}<r_{1}r_{2}$. Moreover for all $x\in S_{\lambda}$ such
that $\abs{wx^{a}}<r_{1}r_{2}$, Cauchy's formula gives 
\begin{eqnarray*}
\Psi_{w,\lambda,n}\left(w\right) & = & \frac{1}{2i\pi}\oint_{\gamma\left(x,w\right)}\frac{\Psi_{w,\lambda}\left(h_{1},\frac{w}{h_{1}}\right)}{h_{1}^{n+1}}\mbox{d}h_{1}\,\,,\mbox{ for all }n\in\wn,
\end{eqnarray*}
where $\gamma\left(x,w\right)$ is any circle (oriented positively)
centered at the origin with a radius $\rho\left(x,w\right)$ satisfying
\[
\abs{\frac{wf_{2}\left(x,w\right)}{r_{2}}}<\rho\left(x,w\right)<\abs{\frac{r_{1}}{f_{1}\left(x,w\right)}}\,\,.
\]
If $\abs{wx^{a}}<\left(\tilde{r}_{1}+\delta\right)\left(\tilde{r}_{2}+\delta\right)$,
we can take for instance 
\[
\rho\left(x,w\right)=\abs{\frac{\tilde{r}_{1}+\delta}{f_{1}\left(x,w\right)}}\,\,.
\]
Therefore, for all $x\in S_{\lambda}$ and all $w\in\ww C$ such that
$\abs{wx^{a}}\leq\tilde{r}_{1}\tilde{r}_{2}$, for all $\xi\in\ww C$
with $\abs{\xi}<\delta$, we also have: 
\begin{eqnarray*}
\Psi_{w,\lambda,n}\left(w+\xi\right) & = & \frac{1}{2i\pi}\oint_{\gamma\left(x,w\right)}\frac{\Psi_{w,\lambda}\left(h_{1},\frac{w+\xi}{h_{1}}\right)}{h_{1}^{n+1}}\mbox{d}h_{1}\,\,,\mbox{ for all }n\in\ww Z,
\end{eqnarray*}
where $\gamma\left(x,w\right)$ is the same circle $\Big($of radius
${\displaystyle \rho\left(x,w\right)=\abs{\frac{\tilde{r}_{1}+\delta}{f_{1}\left(x,w\right)}}}$$\Big)$
for all $\abs{\xi}<\delta$. Moreover, since for all $x\in S_{\lambda}$,
we have 
\[
\Psi_{\lambda}\left(\lsr{}{x,\mathbf{r}}\right)\subset\lsr{}{x,\mathbf{r'}}\,\,,
\]
and since for all $\left(h'_{1},h'_{2}\right)\in\lsr{}{x,\mathbf{r'}}$
we have 
\[
\abs{h'_{1}h'_{2}}\leq\frac{r'_{1}r'_{2}}{\abs{x^{a}}}\,\,,
\]
then for all $x\in S_{\lambda}$ and $w\in\ww C$ such that $\abs{wx^{a}}\leq\tilde{r}_{1}\tilde{r}_{2}$,
the following inequality holds for all $h_{1}$ with $\abs{h_{1}}<\frac{r_{1}}{f_{1}\left(x,w\right)}$:
\begin{eqnarray*}
\abs{\Psi_{w,\lambda}\left(h_{1},\frac{w}{h_{1}}\right)} & < & \frac{r'_{1}r'_{2}}{\abs{x^{a}}}\,.
\end{eqnarray*}
The well-known theorem regarding integrals depending analytically
on a parameter asserts that for all $n\in\ww Z$ the mapping $\Psi_{w,\lambda,n}$
is analytic near any point $w\in\ww C$. Hence, it is an entire function
(\emph{i.e.} analytic over $\wc$). Moreover, the inequality above
and the Cauchy's formula together imply that for all $n\in\ww Z$
and for all $\left(x,w\right)\in S_{\lambda}\times\ww C$ such that
$\abs{wx^{a}}\leq\tilde{r}_{1}\tilde{r}_{2}$, we have: 
\begin{eqnarray*}
\abs{\Psi_{w,\lambda,n}\left(w\right)} & < & \frac{r'_{1}r'_{2}}{\abs{x^{a}}\rho\left(x,w\right)^{n}}=\frac{r'_{1}r'_{2}}{\abs{x^{a}}}\abs{\frac{f_{1}\left(x,w\right)}{\tilde{r}_{1}+\delta}}^{n}\,.
\end{eqnarray*}
According to Fact \ref{fact: limites}, for a fixed value $w\in\ww C$,
if $n<0$, the right hand-side tends to $0$ as $x$ tends to $0$
in $S_{\lambda}$. This implies in particular that $\Psi_{w,\lambda,n}=0$
for all $n<0$. Consequently: 
\begin{eqnarray*}
\Psi_{w,\lambda}\left(h_{1},\frac{w}{h_{1}}\right) & = & \sum_{n\geq0}\Psi_{w,\lambda,n}\left(w\right)h_{1}^{n}\,\,.
\end{eqnarray*}
Moreover, for all $w\in\ww C$ the series converges normally in every
domain of the form 
\[
\Omega_{x,w}:=\acc{h_{1}\in\ww C\mid\abs{h_{1}}\leq\abs{\frac{\tilde{r}_{1}}{f_{1}\left(x,w\right)}}}\,\,,\mbox{ for all }x\in S_{\lambda}\,\,,\,0<\tilde{r}_{1}<r_{1},
\]
since the Laurent expansion's range is $n\geq0$. This actually means
that the series converges normally in an entire neighborhood of the
origin in $\ww C$. In particular, for all fixed $w\in\ww C$, the
map 
\[
h_{1}\mapsto\Psi_{w,\lambda}\left(h_{1},\frac{w}{h_{1}}\right)=\sum_{n\geq0}\Psi_{w,\lambda,n}\left(w\right)h_{1}^{n}
\]
 is analytic in a neighborhood of the origin. Finally, the series
\begin{eqnarray*}
\Psi_{w,\lambda}\left(h_{1},h_{2}\right) & = & \sum_{n\geq0}\Psi_{w,\lambda,n}\left(h_{1}h_{2}\right)h_{1}^{n}
\end{eqnarray*}
converges normally, and hence its sum is analytic in every domain
of the form $\lsr{}{\tilde{\mathbf{r}}}$, with $0<\tilde{r}_{1}<r_{1}$
and $0<\tilde{r}_{2}<r_{2}$.
\end{proof}

\subsection{Action on the resonant monomial}

~

Since $\psi_{\pm\lambda}\in\Lambda_{\pm\lambda}^{\left(\tx{weak}\right)}\left(\ynorm\right)$,
the mapping $\psi_{\pm\lambda}$ is of the form 
\begin{eqnarray*}
\psi_{\pm\lambda}\left(x,\mathbf{y}\right) & = & \left(x,\psi_{1,\pm\lambda}\left(x,\mathbf{y}\right),\psi_{2,\pm\lambda}\left(x,\mathbf{y}\right)\right)\,\,,
\end{eqnarray*}
with $\psi_{1,\pm\lambda},\psi_{2,\pm\lambda}$ analytic and bounded
in $S_{\pm\lambda}\times\mathbf{D\left(0,r\right)}$. Moreover, by
assumption $\psi_{\pm\lambda}$ admits the identity as weak Gevrey-1
asymptotic expansion, \emph{i.e. }we have a normally convergent expansion:
\begin{eqnarray*}
\psi_{i,\pm\lambda}\left(x,\mathbf{y}\right) & = & y_{i}+\sum_{\mathbf{k}\in\ww N^{2}}\psi_{i,\mathbf{\pm\lambda,k}}\left(x\right)\mathbf{y^{k}}\,\,,
\end{eqnarray*}
where $\psi_{i,\pm\lambda,\mathbf{k}}$ is holomorphic in $S_{\pm\lambda}$
and admits $0$ as Gevrey-1 asymptotic expansion, for $i=1,2$ and
all $\mathbf{k}=\left(k_{1},k_{2}\right)\in\ww N^{2}$.
\begin{lem}
\label{lem: istropie monomiale}With the notations and assumptions
above, let us define $\psi_{v,\pm\lambda}:=\psi_{1,\pm\lambda}\psi_{2,\pm\lambda}$.
Then $\psi_{v,\lambda}$ and $\psi_{v,-\lambda}$ can be expanded
as the series 
\[
\begin{cases}
{\displaystyle \psi_{v,\lambda}\left(x,\mathbf{y}\right)=y_{1}y_{2}+x^{a}\sum_{n\ge1}\Psi_{w,\lambda,n}\left(\frac{y_{1}y_{2}}{x^{a}}\right)\left(\frac{y_{1}}{f_{1}\left(x,\frac{y_{1}y_{2}}{x^{a}}\right)}\right)^{n}}\\
{\displaystyle \psi_{v,-\lambda}\left(x,\mathbf{y}\right)=y_{1}y_{2}+x^{a}\sum_{n\ge1}\Psi_{w,-\lambda,n}\left(\frac{y_{1}y_{2}}{x^{a}}\right)\left(\frac{y_{2}}{f_{2}\left(x,\frac{y_{1}y_{2}}{x^{a}}\right)}\right)^{n}}
\end{cases}
\]
which are normally convergent in every subset of $S_{\pm\lambda}\times\mathbf{D\left(0,r\right)}$
of the form $S_{\pm\lambda}\times\mathbf{\overline{D}\left(0,\tilde{r}\right)},$
where $\mathbf{\overline{D}\left(0,\tilde{r}\right)}$ is a closed
poly-disc with $\mathbf{\tilde{r}}=\left(\tilde{r}_{1},\tilde{r}_{2}\right)$
such that 
\[
0<\tilde{r}_{j}<r_{j}~~~,~j\in\left\{ 1,2\right\} .
\]
Here $\Psi_{w,\lambda,n}$ and $\Psi_{w,-\lambda,n}$ , for $n\in\ww N$,
are the ones appearing in Lemma \ref{lem: isotropie monomone res esp des feuilles}.
Moreover, for all closed sub-sector $S'\subset S_{\pm\lambda}$ and
for all closed poly-disc $\mathbf{\overline{D}}\subset\mathbf{D\left(0,r\right)}$,
there exists $A,B>0$ such that:
\begin{eqnarray*}
\abs{\psi_{v,\pm\lambda}\left(x,y_{1},y_{2}\right)-y_{1}y_{2}} & \leq & A\exp\left(-\frac{B}{\abs x}\right)\,\,,\,\,\forall\left(x,\mathbf{y}\right)\in S'\times\mathbf{\overline{D}}\,\,.
\end{eqnarray*}
In particular, $\psi_{v,\pm\lambda}$ admits $y_{1}y_{2}$ as Gevrey-1
asymptotic expansion in $S_{\pm\lambda}\times\mathbf{D\left(0,r\right)}$.\end{lem}
\begin{proof}
By definition, we have 
\[
\Psi_{\pm\lambda}\circ\cal H_{\pm\lambda}=\cal H_{\pm\lambda}\circ\psi_{\pm\lambda}\,.
\]
In particular, for all $\left(x,\mathbf{y}\right)\in S_{\pm\lambda}\times\mathbf{D\left(0,r\right)}$:
\begin{eqnarray*}
\Psi_{w,\pm}\left(x,\frac{y_{1}}{f_{1}\left(x,\frac{y_{1}y_{2}}{x^{a}}\right)},\frac{y_{2}}{f_{2}\left(x,\frac{y_{1}y_{2}}{x^{a}}\right)}\right) & = & \frac{\psi_{v,\pm}\left(x,y_{1},y_{2}\right)}{x^{a}}\,\,.
\end{eqnarray*}
Thus, according to Lemma \ref{lem: isotropie monomone res esp des feuilles}
we have: 
\begin{equation}
\begin{cases}
{\displaystyle \psi_{v,\lambda}\left(x,\mathbf{y}\right)=x^{a}\sum_{n\ge0}\Psi_{w,\lambda,n}\left(\frac{y_{1}y_{2}}{x^{a}}\right)\left(\frac{y_{1}}{f_{1}\left(x,\frac{y_{1}y_{2}}{x^{a}}\right)}\right)^{n}}\\
{\displaystyle \psi_{v,-\lambda}\left(x,\mathbf{y}\right)=x^{a}\sum_{n\ge0}\Psi_{w,-\lambda,n}\left(\frac{y_{1}y_{2}}{x^{a}}\right)\left(\frac{y_{2}}{f_{2}\left(x,\frac{y_{1}y_{2}}{x^{a}}\right)}\right)^{n}} & \,.
\end{cases}\label{eq: psi_v et Psi_v}
\end{equation}
Besides we know that $\psi_{v,\pm\lambda}$ admits $y_{1}y_{2}$ as
weak Gevrey-1 asymptotic expansion in $S_{\pm\lambda}\times\mathbf{D\left(0,r\right)}$:
\begin{eqnarray}
\psi_{v,\pm\lambda}\left(x,y_{1},y_{2}\right) & = & y_{1}y_{2}+\sum_{\mathbf{k}\in\ww N^{2}}\psi_{v,\pm\lambda,\mathbf{k}}\left(x\right)\mathbf{y^{k}}\,\,,\label{eq: psi_v developpement}
\end{eqnarray}
where for all $\mathbf{k}=\left(k_{1},k_{2}\right)\in\ww N^{2}$ the
mapping $\psi_{v,\pm\lambda,\mathbf{k}}$ is holomorphic in $S_{\pm\lambda}$
and admits $0$ as Gevrey-1 asymptotic expansion. Let us compare both
expressions of $\psi_{v,\pm\lambda}$ above. Looking at monomials
$\mathbf{y^{k}}$ with $k_{1}=k_{2}$ in (\ref{eq: psi_v developpement}),
and at terms corresponding to $n=0$ on the right-hand side of (\ref{eq: psi_v et Psi_v}),
we must have for all $x\in S_{\pm\lambda}$ and $v\in\ww C$ with
$\abs v<r_{1}r_{2}$:
\begin{eqnarray*}
v+\sum_{k\ge0}\psi_{v,\lambda,\mathbf{k}\left(k,k\right)}\left(x\right)v^{k} & = & x^{a}\Psi_{w,\lambda,0}\left(\frac{v}{x^{a}}\right)\,\,.
\end{eqnarray*}
Since $\Psi_{w,\pm\lambda,0}$ is analytic in $\ww C$, there exists
$\left(\alpha_{\text{\ensuremath{\pm\lambda},}k}\right)_{k\in\ww N}\subset\ww C$
such that
\begin{eqnarray*}
\Psi_{w,\pm\lambda,0}\left(\frac{v}{x^{a}}\right) & = & \sum_{k\ge0}\alpha_{\pm\lambda,k}\left(\frac{v}{x^{a}}\right)^{k}\,\,.
\end{eqnarray*}
This can only happen if $\alpha_{\pm\lambda,k}=0$ whenever $k\neq1$,
for $\psi_{v,\pm\lambda,\mathbf{k}}$ is holomorphic in $S_{\pm\lambda}$
and admits $0$ as Gevrey-1 asymptotic expansion. A further immediate
identification yields 
\begin{align*}
\Psi_{v,\pm\lambda,0}\left(w\right) & =w~.
\end{align*}
Thus
\[
\begin{cases}
{\displaystyle \psi_{v,\lambda}\left(x,\mathbf{y}\right)=y_{1}y_{2}+x^{a}\sum_{n\ge1}\Psi_{w,\lambda,n}\left(\frac{y_{1}y_{2}}{x^{a}}\right)\left(\frac{y_{1}}{f_{1}\left(x,\frac{y_{1}y_{2}}{x^{a}}\right)}\right)^{n}}\\
{\displaystyle \psi_{v,-\lambda}\left(x,\mathbf{y}\right)=y_{1}y_{2}+x^{a}\sum_{n\ge1}\Psi_{w,-\lambda,n}\left(\frac{y_{1}y_{2}}{x^{a}}\right)\left(\frac{y_{2}}{f_{2}\left(x,\frac{y_{1}y_{2}}{x^{a}}\right)}\right)^{n}} & \,\,.
\end{cases}
\]

Let us prove that $\psi_{v,\pm\lambda}$ admits $y_{1}y_{2}$ as Gevrey-1
asymptotic expansion in $S_{\pm\lambda}\times\left(\ww C^{2},0\right)$.
We have to show that ${\displaystyle \abs{\psi_{v,\pm\lambda}\left(x,y_{1},y_{2}\right)-y_{1}y_{2}}}$
is exponentially small with respect to $x\in S_{\pm\lambda}$, uniformly
in $\mathbf{y}\in\mathbf{D\left(0,r\right)}$. As for the previous
lemma, we perform the proof for $\psi_{v,\lambda}$ only (the same
proof applies for $\psi_{v,-\lambda}$ by exchanging $y_{1}$ and
$y_{2}$). 

From the computations above we derive 
\begin{eqnarray*}
\abs{\psi_{v,\lambda}\left(x,y_{1},y_{2}\right)-y_{1}y_{2}} & \leq & \sum_{n\geq1}\abs{x^{a}\Psi_{w,\lambda,n}\left(\frac{y_{1}y_{2}}{x^{a}}\right)\left(\frac{y_{1}}{f_{1}\left(x,\frac{y_{1}y_{2}}{x^{a}}\right)}\right)^{n}}\,\,.
\end{eqnarray*}
Let us fix $\tilde{r}_{1},\tilde{r}_{2},\delta>0$ in such a way that
\[
0<\tilde{r}_{j}+\delta<r_{j}~,~j\in\left\{ 1,2\right\} ~.
\]
Let us take $\abs x$, $\abs{y_{1}}$ and $\abs{y_{2}}$ small enough
so that 
\[
{\displaystyle 2x\in S_{\lambda}}
\]
 and 
\[
{\displaystyle \abs{y_{1}y_{2}}<\frac{\tilde{r}_{1}\tilde{r}_{2}}{\abs{2^{a}}}<r_{1}r_{2}\,\,.}
\]
According to Lemma \ref{lem: isotropie monomone res esp des feuilles},
for all $\tilde{x}\in S_{\lambda}$ and all $w\in\ww C$: 
\begin{eqnarray*}
\abs{w\tilde{x}^{a}}\leq\tilde{r}_{1}\tilde{r}_{2} & \Longrightarrow & \abs{\Psi_{w,\lambda,n}\left(w\right)}\leq\frac{r'_{1}r'_{2}}{\abs{\tilde{x}^{a}}}\abs{\frac{f_{1}\left(\tilde{x},w\right)}{\tilde{r}_{1}+\delta}}^{n}\,\,.
\end{eqnarray*}
In particular for $\tilde{x}=2x$ and ${\displaystyle w}=\frac{y_{1}y_{2}}{x^{a}}$
we derive $\abs{w\tilde{x}^{a}}<\tilde{r}_{1}\tilde{r}_{2}$, from
which we conclude 
\begin{eqnarray*}
\abs{\Psi_{w,\lambda,n}\left(\frac{y_{1}y_{2}}{x^{a}}\right)} & \leq & \frac{r'_{1}r'_{2}}{\abs{2^{a}x^{a}}}\abs{\frac{f_{1}\left(2x,\frac{y_{1}y_{2}}{x^{a}}\right)}{\tilde{r}_{1}+\delta}}^{n}\,\,.
\end{eqnarray*}
Consequently, for all $\left(x,y_{1},y_{2}\right)\in S_{\lambda}\times\mathbf{D\left(0,\tilde{r}\right)}$
with
\[
\begin{cases}
{\displaystyle 2x\in S_{\lambda}}\\
\abs{y_{1}y_{2}}<\frac{\tilde{r}_{1}\tilde{r}_{2}}{\abs{2^{a}}}<r_{1}r_{2} & ,
\end{cases}
\]
we have 
\begin{eqnarray*}
\abs{\psi_{v,\lambda}\left(x,y_{1},y_{2}\right)-y_{1}y_{2}} & \leq & \sum_{n\geq1}\abs{x^{a}\frac{r'_{1}r'_{2}}{2^{a}x^{a}}\left(\frac{f_{1}\left(2x,\frac{y_{1}y_{2}}{x^{a}}\right)}{\tilde{r}_{1}+\delta}\right)^{n}\left(\frac{y_{1}}{f_{1}\left(x,\frac{y_{1}y_{2}}{x^{a}}\right)}\right)^{n}}\\
 & \leq & \frac{r'_{1}r'_{2}}{\abs{2^{a}}}\sum_{n\geq1}\abs{\left(\frac{y_{1}}{\tilde{r}_{1}+\delta}\right)^{n}\left(\frac{f_{1}\left(2x,\frac{y_{1}y_{2}}{x^{a}}\right)}{f_{1}\left(x,\frac{y_{1}y_{2}}{x^{a}}\right)}\right)^{n}}\,\,.
\end{eqnarray*}
Since $\tilde{c}\left(v\right)$ is the germ of an analytic function
near the origin which is null at the origin, we can take $r_{1}$,$r_{2}>0$
small enough in order that for all closed sub-sector $S'\subset S_{\lambda}$
, for all $\tilde{r}_{1}\in\left]0,r_{1}\right[$ and $\tilde{r}_{2}\in\left]0,r_{2}\right[$,
there exist $A,B>0$ satisfying:
\begin{eqnarray*}
\left(x,y_{1},y_{2}\right)\in S'\times\mathbf{D\left(0,\tilde{r}\right)} & \Longrightarrow & \abs{\psi_{v,\lambda}\left(x,y_{1},y_{2}\right)-y_{1}y_{2}}A\exp\left(-\frac{B}{\abs x}\right)\,\,.
\end{eqnarray*}
Let us prove this. We need here to estimate the quantity:
\begin{eqnarray*}
\abs{\frac{f_{1}\left(2x,\frac{y_{1}y_{2}}{x^{a}}\right)}{f_{1}\left(x,\frac{y_{1}y_{2}}{x^{a}}\right)}} & = & \abs{2^{a_{1}}\exp\left(-\frac{\lambda}{2x}-c_{m}\frac{\left(y_{1}y_{2}\right)^{m}}{x}\log\left(2\right)-\frac{\tilde{c}\left(y_{1}y_{2}2^{a}\right)}{2x}+\frac{\tilde{c}\left(y_{1}y_{2}\right)}{x}\right)}\,\,.
\end{eqnarray*}
On only have tot deal with the case where $x\in S'$ is such that
$2x\in S'$ (otherwise, $x$ is ``far from the origin'', and we
conclude without difficulty). We have:
\begin{eqnarray*}
\left(x,y_{1},y_{2}\right)\in S'\times\mathbf{D\left(0,\tilde{r}\right)}\tx{\, et\,}2x\in S & \Longrightarrow & \abs{\frac{f_{1}\left(2x,\frac{y_{1}y_{2}}{x^{a}}\right)}{f_{1}\left(x,\frac{y_{1}y_{2}}{x^{a}}\right)}}\leq\abs{2^{a_{1}}}\exp\left(-\frac{B}{\abs x}\right)<1\,\,.
\end{eqnarray*}
Hence 
\begin{eqnarray*}
\abs{\psi_{v,\lambda}\left(x,y_{1},y_{2}\right)-y_{1}y_{2}} & \leq & \frac{r'_{1}r'_{2}}{\abs{2^{a}}}\sum_{n\geq1}\abs{\frac{2^{a_{1}}y_{1}}{\tilde{r}_{1}+\delta}\exp\left(-\frac{B}{\abs x}\right)}^{n}\\
 & \leq & \frac{r'_{1}r'_{2}}{\abs{2^{a}}}\frac{\abs{\frac{2^{a_{1}}y_{1}}{\tilde{r}_{1}+\delta}\exp\left(-\frac{B}{\abs x}\right)}}{1-\abs{\frac{2^{a_{1}}y_{1}}{\tilde{r}_{1}+\delta}\exp\left(-\frac{B}{\abs x}\right)}}\\
 & \leq & A\exp\left(-\frac{B}{\abs x}\right)\,\,,
\end{eqnarray*}
for a convenient $A>0$. 
\end{proof}
The latter lemma implies $\Psi_{v,\pm\lambda,0}\left(w\right)=w$,
having for consequence the next result.
\begin{cor}
\label{cor: quantit=0000E9 born=0000E9e}For all closed sub-sector
$S'\subset S_{\pm\lambda}$ and for all $\tilde{r}_{1}\in\left]0,r_{1}\right[$
and $\tilde{r}_{2}\in\left]0,r_{2}\right[$, there exists $A,B>0$
such that for all $x\in S'$: 
\begin{eqnarray*}
\left.\begin{array}{c}
{\displaystyle \abs{h_{1}}\leq\frac{\tilde{r}_{1}}{\abs{f_{1}\left(x,h_{1}h_{2}\right)}}}\\
{\displaystyle \abs{h_{2}}\leq\frac{\tilde{r}_{2}}{\abs{f_{2}\left(x,h_{1}h_{2}\right)}}}
\end{array}\right\}  & \Longrightarrow & \abs{\Psi_{w,\pm}\left(x,h_{1},h_{2}\right)-h_{1}h_{1}}\leq\frac{A\exp\left(-\frac{B}{\abs x}\right)}{\abs{x^{a}}}\,\,.
\end{eqnarray*}
In particular, there exists $C>0$ such that: 
\begin{eqnarray*}
\left.\begin{array}{c}
{\displaystyle \abs{h_{1}}\leq\frac{\tilde{r}_{1}}{\abs{f_{1}\left(x,h_{1}h_{2}\right)}}}\\
{\displaystyle \abs{h_{2}}\leq\frac{\tilde{r}_{2}}{\abs{f_{2}\left(x,h_{1}h_{2}\right)}}}
\end{array}\right\}  & \Longrightarrow & \frac{\abs{\exp\left(c_{m}\left(h_{1}h_{2}\right)^{m}\log\left(x\right)+\frac{\tilde{c}\left(x^{a}\left(h_{1}h_{2}\right)^{m}\right)}{x}\right)}}{\abs{\exp\left(c_{m}\left(\Psi_{w}\left(x,h_{1},h_{2}\right)\right)^{m}\log\left(x\right)+\frac{\tilde{c}\left(x^{a}\left(\Psi_{w}\left(x,h_{1},h_{2}\right)\right)^{m}\right)}{x}\right)}}<C\,.
\end{eqnarray*}

\end{cor}

\subsection{Power series expansion of sectorial isotropies in the space of leaves}

~

Now, we give a power series expansion of $\Psi_{1,\pm\lambda}$ and
$\Psi_{2,\pm\lambda}$ in the space of leaves. Let us introduce the
following notations:
\[
\begin{cases}
N\left(1,+\right):=N\left(2,-\right):=1\\
N\left(1,-\right):=N\left(2,+\right):=-1 & .
\end{cases}
\]

\begin{lem}
\label{lem: istropies espace des feuilles 2}With the notations and
assumptions above, there exists entire functions (\emph{i.e. }analytic
over $\ww C$) denoted by $\Psi_{j,\pm\lambda,n}$, $j\in\acc{1,2}$,
$n\geq N\left(j,\pm\right)$, such that for $j\in\acc{1,2}:$
\[
\begin{cases}
{\displaystyle \Psi_{j,\lambda}\left(h_{1},h_{2}\right)=\sum_{n\geq N\left(j,+\right)}\Psi_{j,\lambda,n}\left(h_{1}h_{2}\right)h_{1}^{n}}\\
{\displaystyle \Psi_{j,-\lambda}\left(h_{1},h_{2}\right)=\sum_{n\geq N\left(j,-\right)}\Psi_{j,\lambda,n}\left(h_{1}h_{2}\right)h_{2}^{n}} & .
\end{cases}
\]
These series converge normally in every subset of $\Gamma_{\pm\lambda}$
of the form $\lsr{\pm}{\tilde{\mathbf{r}}}$ with $0<\tilde{r}_{1}<r_{1}$
and $0<\tilde{r}_{2}<r_{2}$ (\emph{cf.} Definition\emph{ }\ref{def: espace des feuilles =0000E0 rayon fixe}).
More precisely, for all $\tilde{r}_{1},\tilde{r}_{2},\delta>0$ such
that 
\[
0<\tilde{r}_{j}+\delta<r_{j}~,~j\in\left\{ 1,2\right\} 
\]
there exists $C>0$ such that for all $x\in S_{\pm\lambda}$ and for
all $w\in\ww C$, we have: 
\begin{eqnarray*}
\abs{wx^{a}}\leq\tilde{r}_{1}\tilde{r}_{2} & \Longrightarrow & \begin{cases}
{\displaystyle \abs{\Psi_{1,\lambda,n}\left(w\right)}<Cr'_{1}\frac{\abs{f_{1}\left(x,w\right)}^{n-1}}{\left(\tilde{r}_{1}+\delta\right)^{n}}} & ,\, n\geq1\\
{\displaystyle \abs{\Psi_{2,\lambda,n}\left(w\right)}<\frac{Cr'_{2}}{\abs{x^{a}}}\frac{\abs{f_{1}\left(x,w\right)}^{n+1}}{\left(\tilde{r}_{1}+\delta\right)^{n}}} & ,\, n\geq-1\\
{\displaystyle \abs{\Psi_{1,-\lambda,n}\left(w\right)}<\frac{Cr'_{1}}{\abs{x^{a}}}\frac{\abs{f_{2}\left(x,w\right)}^{n+1}}{\left(\tilde{r}_{2}+\delta\right)^{n}}} & ,\, n\geq-1\\
{\displaystyle \abs{\Psi_{2,-\lambda,n}\left(w\right)}<Cr'_{2}\frac{\abs{f_{2}\left(x,w\right)}^{n-1}}{\left(\tilde{r}_{2}+\delta\right)^{n}}} & ,\, n\geq1\,\,.
\end{cases}
\end{eqnarray*}
Moreover: 
\[
\Psi_{1,-\lambda,-1}\left(0\right)=\Psi_{2,\lambda,-1}\left(0\right)=0\,\,.
\]
\end{lem}
\begin{proof}
We use the same notations as in the proof of Lemma \ref{lem: isotropie monomone res esp des feuilles},
and as usual, we give the proof only for $\Psi_{\lambda}$ (the proof
for $\Psi_{-\lambda}$ is analogous, by exchanging the role played
by $h_{1}$ and $h_{2}$). For fixed $w\in\ww C$, the maps
\[
\varphi_{1}:h_{1}\mapsto\Psi_{1,\lambda}\left(h_{1},\frac{w}{h_{1}}\right)
\]
and 
\[
\varphi_{2}:h_{1}\mapsto\Psi_{2,\lambda}\left(h_{1},\frac{w}{h_{1}}\right)
\]
are analytic in 
\[
M_{w,1}=\bigcup_{\substack{x\in S_{\lambda}\\
\abs{wx^{a}}<r_{1}r_{2}
}
}\Omega_{x,w}
\]
(see the proof of Lemma \ref{lem: isotropie monomone res esp des feuilles}).
In particular, $\varphi_{1}$ and $\varphi_{2}$ admit Laurent expansions
\[
\begin{cases}
{\displaystyle \varphi_{1}\left(h_{1}\right)=\Psi_{1,\lambda}\left(h_{1},\frac{w}{h_{1}}\right)=\sum_{n\geq-L_{1}}\Psi_{1,\lambda,n}\left(w\right)h_{1}^{n}}\\
{\displaystyle \varphi_{2}\left(h_{1}\right)=\Psi_{2,\lambda}\left(h_{1},\frac{w}{h_{1}}\right)=\sum_{n\geq-L_{2}}\Psi_{2,\lambda,n}\left(w\right)h_{1}^{n}}
\end{cases}
\]
in every annulus $\Omega_{x,w}$, with $x\in S_{\lambda}$ such that
$\abs{wx^{a}}<r_{1}r_{2}$. Using the same method as in the proof
of Lemma \ref{lem: isotropie monomone res esp des feuilles}, we prove
without additional difficulties that for all $n\in\ww Z$, $\Psi_{1,\lambda,n}$
and $\Psi_{2,\lambda,n}$ are analytic in any point $w\in\ww C$,
and thus are entire functions (\emph{i.e.} analytic over $\wc$).
Moreover, we also show in the same way as earlier that for all $\tilde{r}_{1},\tilde{r}_{2},\delta>0$
with 
\[
0<\tilde{r}_{j}+\delta<r_{j}~,~j\in\left\{ 1,2\right\} \,\,,
\]
for all $n\in\ww Z$ and for all $\left(x,w\right)\in S_{\lambda}\times\ww C$
such that $\abs{wx^{a}}\leq\tilde{r}_{1}\tilde{r}_{2}$, we have:
\[
\begin{cases}
{\displaystyle \abs{\Psi_{1,\lambda,n}\left(w\right)}<\frac{r'_{1}}{\abs{f_{1}\left(x,\Psi_{w,\lambda}\left(x,h_{1},\frac{w}{h_{1}}\right)\right)}}\abs{\frac{f_{1}\left(x,w\right)}{\tilde{r}_{1}+\delta}}^{n}}\\
{\displaystyle \abs{\Psi_{2,\lambda,n}\left(w\right)}<\frac{r'_{2}}{\abs{f_{2}\left(x,\Psi_{w,\lambda}\left(x,h_{1},\frac{w}{h_{1}}\right)\right)}}\abs{\frac{f_{1}\left(x,w\right)}{\tilde{r}_{1}+\delta}}^{n}} & \,.
\end{cases}
\]
According to Corollary \ref{cor: quantit=0000E9 born=0000E9e}, there
exists $C>0$ such that for all $\left(x,w\right)\in S_{\lambda}\times\ww C$
with $\abs{wx^{a}}\leq\tilde{r}_{1}\tilde{r}_{2}$, we have: 
\[
\begin{cases}
{\displaystyle \abs{\Psi_{1,\lambda,n}\left(w\right)}<Cr'_{1}\frac{\abs{f_{1}\left(x,w\right)}^{n-1}}{\left(\tilde{r}_{1}+\delta\right)^{n}}}\\
{\displaystyle \abs{\Psi_{2,\lambda,n}\left(w\right)}<\frac{Cr'_{2}}{\abs{x^{a}}}\frac{\abs{f_{1}\left(x,w\right)}^{n+1}}{\left(\tilde{r}_{1}+\delta\right)^{n}}} & \,.
\end{cases}
\]
According to the statement in Fact \ref{fact: limites}, for a fixed
value $w\in\ww C$, if we look at the limit as $x$ tends to $0$
in $S_{\lambda}$ of the right hand-sides above we deduce that: 
\[
\begin{cases}
\abs{\Psi_{1,\lambda,n}\left(w\right)}=0 & ,\,\forall n\leq0\\
\abs{\Psi_{2,\lambda,n}\left(w\right)}=0 & ,\,\forall n\leq-2\,.
\end{cases}
\]
Consequently: 
\[
\begin{cases}
{\displaystyle \Psi_{1,\lambda}\left(h_{1},h_{2}\right)=\sum_{n\geq1}\Psi_{1,\lambda,n}\left(h_{1}h_{2}\right)h_{1}^{n}}\\
{\displaystyle \Psi_{2,\lambda}\left(h_{1},h_{2}\right)=\sum_{n\geq-1}\Psi_{2,\lambda,n}\left(h_{1}h_{2}\right)h_{1}^{n}} & \,\,.
\end{cases}
\]
These function series converges normally (and are analytic) in every
domain of the form$\lsr{}{\tilde{\mathbf{r}}}$ with $\tilde{\mathbf{r}}:=\left(\tilde{r}_{1},\tilde{r}_{2}\right)$
and 
\[
0<\tilde{r}_{j}+\delta<r_{j}~,~j\in\left\{ 1,2\right\} \,\,
\]
(\emph{cf.} Definition\emph{ }\ref{def: espace des feuilles =0000E0 rayon fixe}).
Moreover, for any fixed value of $h_{2}$, on the one hand the function
series 
\[
h_{1}\mapsto\Psi_{2,\lambda}\left(h_{1},h_{2}\right)=\sum_{n\geq-1}\Psi_{2,\lambda,n}\left(h_{1}h_{2}\right)h_{1}^{n}
\]
is analytic in a punctured disc, since 
\begin{eqnarray*}
{\displaystyle \abs{f_{2}\left(x,h_{1},h_{2}\right)}} & \underset{\substack{x\rightarrow0\\
x\in S_{\lambda}
}
}{\longrightarrow} & 0\,,
\end{eqnarray*}
and on the other hand, we already know that the function $h_{1}\mapsto\Psi_{2,\lambda}\left(h_{1},h_{2}\right)$
is analytic in a neighborhood of the origin. Thus, we must have $\Psi_{2,\lambda,-1}\left(0\right)=0$.
\end{proof}

\subsection{Sectorial isotropies: proof of Proposition \ref{prop: isotropies plates}}

~

The following lemma is a more precise version of Proposition \ref{prop: isotropies plates}.
We recall the notations:
\[
\begin{cases}
N\left(1,+\right)=N\left(2,-\right)=1\\
N\left(1,-\right)=N\left(2,+\right)=-1 & .
\end{cases}
\]

\begin{lem}
\label{lem: isotropies exp plates}With the notations and assumptions
above, we consider $\psi_{\pm\lambda}\in\Lambda_{\pm\lambda}^{\left(\tx{weak}\right)}\left(\ynorm\right)$,
with
\begin{eqnarray*}
\psi_{\pm\lambda}\left(x,\mathbf{y}\right) & = & \left(x,\psi_{1,\pm\lambda}\left(x,\mathbf{y}\right),\psi_{2,\pm\lambda}\left(x,\mathbf{y}\right)\right)\,\,.
\end{eqnarray*}
 Then, for $i\in\acc{1,2}$, $\psi_{i,\lambda}$ and $\psi_{i,-\lambda}$
can be written as power series as follows: 
\[
\begin{cases}
{\displaystyle \psi_{i,\lambda}\left(x,\mathbf{y}\right)=y_{i}+f_{i}\left(x,\frac{\psi_{v,\lambda}\left(x,\mathbf{y}\right)}{x^{a}}\right)\sum_{n\ge N\left(i,+\right)+1}\Psi_{i,\lambda,n}\left(\frac{y_{1}y_{2}}{x^{a}}\right)\left(\frac{y_{1}}{f_{1}\left(x,\frac{y_{1}y_{2}}{x^{a}}\right)}\right)^{n}}\\
{\displaystyle \psi_{i,-\lambda}\left(x,\mathbf{y}\right)=y_{i}+f_{i}\left(x,\frac{\psi_{v,-\lambda}\left(x,\mathbf{y}\right)}{x^{a}}\right)\sum_{n\ge N\left(i,-\right)+1}\Psi_{i,-\lambda,n}\left(\frac{y_{1}y_{2}}{x^{a}}\right)\left(\frac{y_{2}}{f_{2}\left(x,\frac{y_{1}y_{2}}{x^{a}}\right)}\right)^{n}} & .
\end{cases}
\]
which are normally convergent in every subset of $S_{\pm\lambda}\times\mathbf{D\left(0,r\right)}$
of the form $S_{\pm\lambda}\times\mathbf{\overline{D}\left(0,\tilde{r}\right)},$
where $\mathbf{\overline{D}\left(0,\tilde{r}\right)}$ is a closed
poly-disc with $\mathbf{\tilde{r}}=\left(\tilde{r}_{1},\tilde{r}_{2}\right)$
such that 
\[
0<\tilde{r}_{j}<r_{j}~~~~,~j\in\left\{ 1,2\right\} ~.
\]
Here $\Psi_{i,\lambda,n}$ , $\Psi_{i,-\lambda,n}$ (for $i=1,2$
and $n\in\ww N$) are given in Lemma \ref{lem: istropies espace des feuilles 2}.
Moreover, for all closed sub-sector $S'\subset S_{\pm\lambda}$ and
for all closed poly-disc $\mathbf{\overline{D}}\subset\mathbf{D\left(0,r\right)}$,
there exists $A,B>0$ such that for $j=1,2$:
\begin{eqnarray*}
\abs{\psi_{j,\pm\lambda}\left(x,y_{1},y_{2}\right)-y_{j}} & \leq & A\exp\left(-\frac{B}{\abs x}\right)\,\,,\,\,\forall\left(x,\mathbf{y}\right)\in S'\times\mathbf{\overline{D}}\,\,.
\end{eqnarray*}
As a consequence, $\psi_{j,\pm\lambda}$ admits $y_{j}$ as Gevrey-1
asymptotic expansion in $S_{\pm\lambda}\times\mathbf{D\left(0,r\right)}$.\end{lem}
\begin{rem}
In particular, $\Psi_{1,\lambda,1}\left(w\right)=\Psi_{2,-\lambda,1}\left(w\right)=1$
and $\Psi_{1,-\lambda,-1}\left(w\right)=\Psi_{2,\lambda,-1}\left(w\right)=w$.\end{rem}
\begin{proof}
By definition, we have 
\[
\Psi_{\pm\lambda}\circ\cal H_{\pm\lambda}=\cal H_{\pm\lambda}\circ\psi_{\pm}\,.
\]
In particular, for $j=1,2$ and all $\left(x,\mathbf{y}\right)\in S_{\pm\lambda}\times\mathbf{D\left(0,r\right)}$:
\begin{eqnarray*}
\Psi_{j,\pm\lambda}\left(x,\frac{y_{1}}{f_{1}\left(x,\frac{y_{1}y_{2}}{x^{a}}\right)},\frac{y_{2}}{f_{2}\left(x,\frac{y_{1}y_{2}}{x^{a}}\right)}\right) & = & \frac{\psi_{j,\pm\lambda}\left(x,y_{1},y_{2}\right)}{f_{j}\left(x,\frac{\psi_{v,\pm}\left(x,y_{1},y_{2}\right)}{x^{a}}\right)}\,\,.
\end{eqnarray*}
Thus, according to Lemma \ref{lem: istropies espace des feuilles 2}
we have for $i=1,2$: 
\begin{equation}
\begin{cases}
{\displaystyle \psi_{i,\lambda}\left(x,\mathbf{y}\right)=f_{i}\left(x,\frac{\psi_{v,\lambda}\left(x,\mathbf{y}\right)}{x^{a}}\right)\sum_{n\ge N\left(i,+\right)}\Psi_{i,\lambda,n}\left(\frac{y_{1}y_{2}}{x^{a}}\right)\left(\frac{y_{1}}{f_{1}\left(x,\frac{y_{1}y_{2}}{x^{a}}\right)}\right)^{n}}\\
{\displaystyle \psi_{i,-}\left(x,\mathbf{y}\right)=f_{i}\left(x,\frac{\psi_{v,-\lambda}\left(x,\mathbf{y}\right)}{x^{a}}\right)\sum_{n\ge N\left(i,-\right)}\Psi_{i,-\lambda,n}\left(\frac{y_{1}y_{2}}{x^{a}}\right)\left(\frac{y_{2}}{f_{2}\left(x,\frac{y_{1}y_{2}}{x^{a}}\right)}\right)^{n}} & ,
\end{cases}\label{eq:psi et Psi}
\end{equation}
and these series are normally convergent (and then define analytic
functions) in any domain of the form $S'\times\mathbf{\overline{D}\left(0,\tilde{r}\right)},$
where $S'$ is a closed sub-sector of $S_{\pm\lambda}$ and $\mathbf{\overline{D}\left(0,\tilde{r}\right)}$
is a closed poly-disc with $\mathbf{\tilde{r}}=\left(\tilde{r}_{1},\tilde{r}_{2}\right)$
such that 
\[
0<\tilde{r}_{j}<r_{j}~~~~,j\in\left\{ 1,2\right\} ~.
\]
Let us compare the different expressions of $\psi_{j,\pm\lambda}$,
$j=1,2$. We know that $\psi_{j,\pm\lambda}\left(x,y_{1},y_{2}\right)$
admits $y_{j}$ as weak Gevrey-1 asymptotic expansion in $S_{\pm\lambda}\times\mathbf{D\left(0,r\right)}$.
Thus, we can write:
\begin{eqnarray*}
\psi_{j,\pm\lambda}\left(x,y_{1},y_{2}\right) & = & y_{j}+\sum_{\mathbf{k}\in\ww N^{2}}\psi_{j,\pm\lambda,\mathbf{k}}\left(x\right)\mathbf{y^{k}}\,\,,
\end{eqnarray*}
where for all $\mathbf{k}=\left(k_{1},k_{2}\right)\in\ww N^{2}$,
$\psi_{j,\pm\lambda,\mathbf{k}}$ is analytic in $S_{\pm\lambda}$
and admits $0$ as Gevrey-1 asymptotic expansion. As usual, let us
deal with the case of $\psi_{1,\lambda}$ and $\psi_{2,\lambda}$
(the other one being similar by exchanging $y_{1}$ and $y_{2}$).

According to the expressions given by Lemmas \ref{lem: isotropie monomone res esp des feuilles}
and \ref{lem: istropies espace des feuilles 2}, we can be more precise
on the index sets in the sums above:
\begin{equation}
\begin{cases}
{\displaystyle \psi_{1,\lambda}\left(x,y_{1},y_{2}\right)=y_{1}+\sum_{\substack{\mathbf{k}=\left(k_{1},k_{2}\right)\in\ww N^{2}\\
k_{1}\geq k_{2}+1
}
}\psi_{1,\lambda,\mathbf{k}}\left(x\right)y_{1}^{k_{1}}y_{2}^{k_{2}}}\\
\psi_{2,\lambda}\left(x,y_{1},y_{2}\right)=y_{2}+\sum_{\substack{\mathbf{k}=\left(k_{1},k_{2}\right)\in\ww N^{2}\\
k_{1}\geq k_{2}
}
}\psi_{2,\lambda,\mathbf{k}}\left(x\right)y_{1}^{k_{1}}y_{2}^{k_{2}} & {\displaystyle .}
\end{cases}\label{eq: psi_1 et psi_2}
\end{equation}
 Let us deal with $\psi_{1,\lambda}$ (a similar proof holds for $\psi_{2,\lambda}$).
Looking at terms for $n=1$ in (\ref{eq:psi et Psi}) and at monomials
terms $\mathbf{y^{k}}$ such that $k_{1}\leq k_{2}+1$ in (\ref{eq: psi_1 et psi_2}),
we must have for all $x\in S_{\lambda}$, $y_{1},y_{2}\in\ww C$ with
$\abs{y_{1}}<r_{1}$, $\abs{y_{2}}<r_{2}$:
\begin{eqnarray*}
1+\sum_{k\ge0}\psi_{1,\lambda,\left(k+1,k\right)}\left(x\right)y_{1}^{k}y_{2}^{k} & = & \frac{f_{1}\left(x,\frac{\psi_{v,\lambda}\left(x,\mathbf{y}\right)}{x^{a}}\right)}{f_{1}\left(x,\frac{y_{1}y_{2}}{x^{z}}\right)}\Psi_{1,\lambda,1}\left(\frac{y_{1}y_{2}}{x^{a}}\right)\,\,.
\end{eqnarray*}
According to Lemma \ref{lem: istropie monomiale} and Corollary\ref{cor: quantit=0000E9 born=0000E9e},
we have: 
\begin{eqnarray*}
\frac{f_{1}\left(x,\frac{\psi_{v,\lambda}\left(x,\mathbf{y}\right)}{x^{a}}\right)}{f_{1}\left(x,\frac{y_{1}y_{2}}{x^{z}}\right)} & = & 1+\sum_{\substack{j_{1}\geq j_{2}+1\geq1}
}F_{j_{1},j_{2}}\left(x\right)y_{1}^{j_{1}}y_{2}^{j_{2}}\\
 & = & 1+\underset{\substack{\left(x,\mathbf{y}\right)\longrightarrow0\\
\left(x,\mathbf{y}\right)\in S_{\lambda}\times\mathbf{D\left(0,r\right)}
}
}{\tx O}\left(\abs{y_{1}}\right)\,\,,
\end{eqnarray*}
for some analytic and bounded functions $F_{j_{1},j_{2}}\left(x\right)$,
$j_{1}\ge j_{2}$. As in the proof of Lemma \ref{lem: istropie monomiale},
using the fact that $\psi_{\lambda}$ admits the identity as weak
Gevrey-1 asymptotic expansion, we deduce that $\Psi_{1,\lambda,1}\left(w\right)=1$,
and then:
\begin{eqnarray*}
\psi_{1,\lambda}\left(x,\mathbf{y}\right) & = & y_{1}+f_{1}\left(x,\frac{\psi_{v,\lambda}\left(x,\mathbf{y}\right)}{x^{a}}\right)\sum_{n\ge2}\Psi_{1,\lambda,n}\left(\frac{y_{1}y_{2}}{x^{a}}\right)\left(\frac{y_{1}}{f_{1}\left(x,\frac{y_{1}y_{2}}{x^{a}}\right)}\right)^{n}\\
 & = & y_{1}+\sum_{\substack{\mathbf{k}=\left(k_{1},k_{2}\right)\in\ww N^{2}\\
k_{1}\geq k_{2}+2
}
}\psi_{1,\lambda,\mathbf{k}}\left(x\right)y_{1}^{k_{1}}y_{2}^{k_{2}}\,\,.
\end{eqnarray*}
It remains to show that $\psi_{1,\lambda}$ admits $y_{1}$ as Gevrey-1
asymptotic expansion in $S_{\lambda}\times\mathbf{D\left(0,r\right)}$.
From the computations above, we deduce: 
\begin{eqnarray*}
\abs{\psi_{1,\lambda}\left(x,y_{1},y_{2}\right)-y_{1}} & \leq & \sum_{n\geq2}\abs{\Psi_{1,\lambda,n}\left(\frac{y_{1}y_{2}}{x^{a}}\right)\left(\frac{y_{1}}{f_{1}\left(x,\frac{y_{1}y_{2}}{x^{a}}\right)}\right)^{n-1}\frac{f_{1}\left(x,\frac{\psi_{v,\lambda}\left(x,\mathbf{y}\right)}{x^{a}}\right)}{f_{1}\left(x,\frac{y_{1}y_{2}}{x^{a}}\right)}y_{1}}\,\,.
\end{eqnarray*}
Using Lemma \ref{lem: istropies espace des feuilles 2}, Corollary
\ref{cor: quantit=0000E9 born=0000E9e} and the same method as at
the end of the proof of Lemma \ref{lem: istropie monomiale}, we can
show the following: we can take $r_{1},r_{2}>0$ small enough such
that for all closed sub-sector $S'$ of $S_{\lambda}$ for all $\tilde{r}_{1}\in\left]0,r_{1}\right[$
and $\tilde{r}_{2}\in\left]0,r_{2}\right[$, there exists $A,B>0$
satisfying:
\begin{eqnarray*}
\left(x,y_{1},y_{2}\right)\in S'\times\mathbf{D\left(0,\tilde{r}\right)} & \Longrightarrow & \abs{\psi_{1,\lambda}\left(x,y_{1},y_{2}\right)-y_{1}}\leq A\exp\left(-\frac{B}{\abs x}\right)\,\,.
\end{eqnarray*}

A similar proof holds for $\psi_{2,\lambda},\psi_{2,-\lambda}$ and
$\psi_{1,-\lambda}$.\end{proof}
\begin{rem}
It should be noticed that in the expressions 
\[
\begin{cases}
{\displaystyle \psi_{1,\lambda}\left(x,\mathbf{y}\right)=y_{1}+f_{1}\left(x,\frac{\psi_{v,\lambda}\left(x,\mathbf{y}\right)}{x^{a}}\right)\sum_{n\ge2}\Psi_{1,\lambda,n}\left(\frac{y_{1}y_{2}}{x^{a}}\right)\left(\frac{y_{1}}{f_{1}\left(x,\frac{y_{1}y_{2}}{x^{a}}\right)}\right)^{n}}\\
{\displaystyle \psi_{1,-\lambda}\left(x,\mathbf{y}\right)=y_{1}+f_{1}\left(x,\frac{\psi_{v,-\lambda}\left(x,\mathbf{y}\right)}{x^{a}}\right)\sum_{n\ge0}\Psi_{1,-\lambda,n}\left(\frac{y_{1}y_{2}}{x^{a}}\right)\left(\frac{y_{2}}{f_{2}\left(x,\frac{y_{1}y_{2}}{x^{a}}\right)}\right)^{n}}
\end{cases}
\]
given by Lemma \ref{lem: isotropies exp plates}, the expansion of
$\psi_{1,\lambda}$ with respect to $\mathbf{y}=\left(y_{1},y_{2}\right)$
starts with a term of order $1$, namely $y_{1}$, followed by terms
of order at least $2$, while in the expansion of $\psi_{1,-\lambda}$,
the term of lowest order is a constant, namely $\Psi_{1,-\lambda,0}\left(0\right)$.
Similarly, the expansion of $\psi_{2,-\lambda}$ (with respect to
$\mathbf{y}=\left(y_{1},y_{2}\right)$) starts with $y_{2}$, while
the expansion of $\psi_{1,-\lambda}$ starts with the constant $\Psi_{2,\lambda,0}\left(0\right)$.
\end{rem}

\section{Description of the moduli space and some applications}

From Lemmas \ref{lem: istropies espace des feuilles 2} and \ref{lem: isotropies exp plates},
we can give a description of the moduli space $\Lambda_{\lambda}\left(\ynorm\right)\times\Lambda_{-\lambda}\left(\ynorm\right)$
of a fixed analytic normal form $\ynorm$.

\subsection{\label{sub: description moduli space}A power series presentation
of the moduli space}

~

We use the notations introduced in section \ref{sec:Sectorial-isotropies-and}.
We denote by $\cal O\left(\text{\ensuremath{\ww C}}\right)$ the set
of entire functions, \emph{i.e.} of functions holomorphic in $\ww C$.
We consider the functions $f_{1}$ and $f_{2}$ defined in $\left(\mbox{\ref{eq: f1 et f2}}\right)$
and introduce four subsets of $\left(\cal O\left(\ww C\right)\right)^{\ww N}$,
denoted by $\cal E_{1,\lambda}\left(\ynorm\right)$, $\cal E_{2,\lambda}\left(\ynorm\right)$,
$\cal E_{1,-\lambda}\left(\ynorm\right)$ and $\cal E_{2,-\lambda}\left(\ynorm\right)$,
defined as follows. On remind the notations 
\[
\begin{cases}
N\left(1,+\right)=N\left(2,-\right)=1\\
N\left(1,-\right)=N\left(2,+\right)=-1 & .
\end{cases}
\]

\begin{defn}
For $j\in\acc{1,2}$, a sequence $\left(\psi_{n}\left(w\right)\right)_{n\geq N\left(j,\pm\right)+1}\in\left(\cal O\left(\ww C\right)\right)^{\ww N}$
belongs to $\cal E_{j,\pm\lambda}\left(\ynorm\right)$ if there exists
an open polydisc $\mathbf{D\left(0,r\right)}$ and an open asymptotic
sector 
\begin{eqnarray*}
S_{\pm\lambda} & \in & \cal{AS}_{\arg\left(\pm\lambda\right),2\pi}
\end{eqnarray*}
such that for all $\tilde{r}_{1},\tilde{r}_{2},\delta>0$ with 
\[
0<\tilde{r}_{i}+\delta<r_{i}~~~~,~i\in\left\{ 1,2\right\} 
\]
there exists $C>0$ such that for all $x\in S_{\lambda}$ (\emph{resp.
$x\in S_{-\lambda}$}) and for all $w\in\ww C$:
\begin{eqnarray*}
\abs{wx^{a}}\leq\tilde{r}_{1}\tilde{r}_{2} & \Longrightarrow & \begin{cases}
{\displaystyle \abs{\psi_{n}\left(w\right)}<C\frac{\abs{f_{1}\left(x,w\right)}^{n-1}}{\left(\tilde{r}_{1}+\delta\right)^{n}}\mbox{ , }\forall n\geq2,} & {\displaystyle \mbox{ if }\left(\psi_{n}\left(w\right)\right)_{n\geq2}\in\cal E_{1,\lambda}\left(\ynorm\right)}\\
{\displaystyle \abs{\psi_{n}\left(w\right)}<\frac{C}{\abs{x^{a}}}\frac{\abs{f_{1}\left(x,w\right)}^{n+1}}{\left(\tilde{r}_{1}+\delta\right)^{n}}\mbox{ , }\forall n\geq0,} & {\displaystyle \mbox{ if }\left(\psi_{n}\left(w\right)\right)_{n\geq2}\in\cal E_{2,\lambda}\left(\ynorm\right)}\\
{\displaystyle \abs{\psi_{n}\left(w\right)}<\frac{C}{\abs{x^{a}}}\frac{\abs{f_{2}\left(x,w\right)}^{n+1}}{\left(\tilde{r}_{2}+\delta\right)^{n}}\mbox{ , }\forall n\geq0,} & {\displaystyle \mbox{ if }\left(\psi_{n}\left(w\right)\right)_{n\geq2}\in\cal E_{1,-\lambda}\left(\ynorm\right)}\\
{\displaystyle \abs{\psi_{n}\left(w\right)}<C\frac{\abs{f_{2}\left(x,w\right)}^{n-1}}{\left(\tilde{r}_{2}+\delta\right)^{n}}\mbox{ , }\forall n\geq2,} & {\displaystyle \mbox{ if }\left(\psi_{n}\left(w\right)\right)_{n\geq2}\in\cal E_{2,-\lambda}\left(\ynorm\right)\,\,.}
\end{cases}
\end{eqnarray*}

\end{defn}
As explained in section \ref{sec:Sectorial-isotropies-and}, we can
associate to any pair 
\[
\left(\psi_{\lambda},\psi_{-\lambda}\right)\in\Lambda_{\lambda}\left(\ynorm\right)\times\Lambda_{-\lambda}\left(\ynorm\right)
\]
two germs of sectorial biholomorphisms of the space of leaves corresponding
to each ``narrow'' sector, which we denote by $\Psi_{\lambda}$
and $\Psi_{-\lambda}$, defined by: 
\begin{equation}
\Psi_{\pm\lambda}:=\cal H_{\pm\lambda}\circ\psi_{\pm\lambda}\circ\cal H_{\pm\lambda}^{-1}\,\,,\label{eq: isotropie dans espace des feuilles}
\end{equation}
where $\cal H_{\pm\lambda}$ is given by Corollary \ref{cor: coordonn=0000E9es espace des feuiles}.
According to Lemmas \ref{lem: istropies espace des feuilles 2} and
\ref{lem: isotropies exp plates}, if we write $\Psi_{\pm\lambda}=\left(x,\Psi_{1,\pm\lambda},\Psi_{2,\pm\lambda}\right)$,
then for $j=1,2$ we have: 
\begin{eqnarray}
{\displaystyle \Psi_{j,\lambda}\left(h_{1},h_{2}\right)} & = & h_{j}+\sum_{n\geq N\left(j,+\right)+1}\Psi_{j,\lambda,n}\left(h_{1}h_{2}\right)h_{1}^{n}\label{eq: description moduli space}\\
{\displaystyle \Psi_{j,-\lambda}\left(h_{1},h_{2}\right)} & = & h_{j}+\sum_{n\geq N\left(j,-\right)+1}\Psi_{j,-\lambda,n}\left(h_{1}h_{2}\right)h_{2}^{n}\nonumber 
\end{eqnarray}
 $\left(\Psi_{j,\pm\lambda,n}\right)_{n}\in{\cal E}_{j,\pm\lambda}$
. Conversely, given $\left(\Psi_{j,\pm\lambda}\right)_{n}\in\cal E_{j,\pm\lambda}$
for $j=1,2$, the estimates made in section \ref{sec:Sectorial-isotropies-and}
show that 
\[
\psi_{\pm\lambda}:=\cal H_{\pm\lambda}^{-1}\circ\Psi_{\pm\lambda}\circ\cal H_{\pm\lambda}~,
\]
where $\Psi_{\pm\lambda}\left(x,\mathbf{h}\right)=\left(x,\Psi_{1,\pm\lambda}\left(\mathbf{h}\right),\Psi_{2,\pm\lambda}\left(\mathbf{h}\right)\right)$,
belongs to $\Lambda_{\pm\lambda}\left(\ynorm\right)$. Consequently,
we can state:
\begin{prop}
\label{prop: description espace de module}We have the following bijections:
\begin{eqnarray*}
\Lambda_{\lambda}\left(\ynorm\right) & \tilde{\rightarrow} & \cal E_{1,\lambda}\left(\ynorm\right)\times\cal E_{2,\lambda}\left(\ynorm\right)\\
\psi_{\lambda} & \mapsto & \left(\Psi_{1,\lambda},\Psi_{2,\lambda}\right)
\end{eqnarray*}
and
\begin{eqnarray*}
\Lambda_{-\lambda}\left(\ynorm\right) & \tilde{\rightarrow} & \cal E_{1,-\lambda}\left(\ynorm\right)\times\cal E_{2,-\lambda}\left(\ynorm\right)\\
\psi_{-\lambda} & \mapsto & \left(\Psi_{1,-\lambda},\Psi_{2,-\lambda}\right)
\end{eqnarray*}
$\Big($we identify here $\Psi_{\pm\lambda}\left(x,\mathbf{h}\right)=\left(x,\Psi_{1,\pm\lambda}\left(\mathbf{h}\right),\Psi_{2,\pm\lambda}\left(\mathbf{h}\right)\right)$
with $\left(\Psi_{1,\pm\lambda}\left(\mathbf{h}\right),\Psi_{2,\pm\lambda}\left(\mathbf{h}\right)\right)$$\Big)$.
\end{prop}

\subsection{Analytic invariant varieties and two-dimensional saddle-nodes}

~

We can give a necessary and sufficient condition for the existence
of analytic invariant varieties in terms of the moduli space described
above.

We recall that for any vector field ${\displaystyle Y\in\cro{\ynorm}}$
as in (\ref{eq: intro}) (\emph{cf.} Definition \ref{def: ynorm class}),
there always exist three formal invariant varieties: $\mathscr{C}=\acc{\left(y_{1},y_{2}\right)=\left(g_{1}\left(x\right),g_{2}\left(x\right)\right)}$,
$\mathscr{H}_{1}=\acc{{\displaystyle y_{1}=f_{1}\left(x,y_{2}\right)}}$
and $\mathscr{H}_{2}=\acc{{\displaystyle y_{2}}=f_{2}\left(x,y_{1}\right)}$,
where $g_{1},g_{2},f_{1},f_{2}$ are formal power series with null
constant term. The first one is classically called the \emph{center
variety, }and we have $\mathscr{C}=\mathscr{H}_{1}\cap\mathscr{H}_{2}$.
If $Y=\ynorm$, then:
\[
\begin{cases}
\mathscr{C}=\acc{y_{1}=y_{2}=0}\\
\mathscr{H}_{1}=\acc{y_{1}=0}\\
\mathscr{H}_{2}=\acc{y_{2}=0} & .
\end{cases}
\]

\begin{prop}
Let ${\displaystyle Y\in\cro{\ynorm}}$ and $\left(\Phi_{\lambda},\Phi_{-\lambda}\right)\in\Lambda_{\lambda}\left(\ynorm\right)\times\Lambda_{-\lambda}\left(\ynorm\right)$
be its Stokes diffeomorphisms. We consider $\Psi_{\pm}=\cal H_{\pm\lambda}\circ\Phi_{\pm\lambda}\circ\cal H_{\pm\lambda}^{-1}$
as above. Then:
\begin{enumerate}
\item the center variety $\mathscr{C}$ is convergent (analytic in the origin)
if and only if $\Psi_{2,\lambda,0}\left(0\right)=\Psi_{1,-\lambda,0}\left(0\right)=0$;
\item the invariant hypersurface $\mathscr{H}_{1}$ is convergent (analytic
in the origin) if and only if for all $n\geq0$, we have $\Psi_{1,-\lambda,n}\left(0\right)=0$;
\item the invariant hypersurface $\mathscr{H}_{2}$ is convergent (analytic
in the origin) if and only if for all $n\geq0$, we have $\Psi_{2,\lambda,n}\left(0\right)=0$.
\end{enumerate}
\end{prop}
\begin{proof}
It is a direct consequence of the power series representation (\ref{eq: description moduli space})
of the Stokes diffeomorphisms $\left(\Phi_{\lambda},\Phi_{-\lambda}\right)$.
Let us explain item $2.$ (the same arguments hold for $1.$ and $3.$
with minor adaptation). The fact that $\Psi_{1,-\lambda,n}\left(0\right)=0$
for all $n\geq0$ means that $\Psi_{1,-\lambda}$ is divisible by
$h_{1}$. Equivalently, both $\Phi_{1,\lambda}$ and $\Phi_{1,-\lambda}$
are divisible by $y_{1}$, so that the analytic hypersurface $\acc{y_{1}=0}$
has the same pre-image by the sectorial normalizing maps $\Phi_{+}$
and $\Phi_{-}$. These pre-images glue together in order to define
an analytic invariant hypersurface $\mathscr{H}_{1}$.
\end{proof}
Notice that if we consider the restriction of a formal normal form
$\ynorm$ to one of the formal invariant hypersurfaces, we obtain
precisely the normal form for two-dimensional saddle-nodes as given
in \cite{MR82}. When one of these hypersurfaces is convergent (\emph{i.e.
}analytic), we recover the Martinet-Ramis invariants by restriction
to this hypersurface, as we present below.
\begin{prop}
Suppose that the formal invariant hypersurface $\mathscr{H}_{1}$
is convergent (\emph{i.e. }analytic in the origin). Then, the Martinet-Ramis
invariants for the saddle-node $Y_{\mid\mathscr{H}_{1}}$ are given
by:
\[
\begin{cases}
{\displaystyle \Psi_{2,\lambda}\left(0,h_{2}\right)=h_{2}+\Psi_{2,\lambda,0}\left(0\right)\in Aff\left(\ww C\right)}\\
{\displaystyle \Psi_{2,-\lambda}\left(0,h_{2}\right)=h_{2}+\sum_{n\geq2}\Psi_{2,-\lambda,n}\left(0\right)h_{2}^{n}\in\diff[\ww C,0]}.
\end{cases}
\]

\end{prop}
Similar result holds for the hypersurface $\mathscr{H}_{2}$.

\subsection{The transversally symplectic case and quasi-linear Stokes phenomena
in the first Painlevé equation }

~

Let us now focus on the transversally symplectic case studied in Theorem
\ref{th: espace de module symplectic}. Let $\ynorm\in\snodiag$ be
transversally symplectic (\emph{i.e.} its residue is $\tx{res}\left(\ynorm\right)=1$).
Using the notations introduced in paragraph \ref{sub: description moduli space},
we define the following sets:
\begin{eqnarray*}
\left(\cal E_{1,\lambda}\left(\ynorm\right)\times\cal E_{2,\lambda}\left(\ynorm\right)\right)_{\omega} & := & \acc{\begin{array}{c}
\Psi_{\lambda}=\left(\Psi_{1,\lambda},\Psi_{2,\lambda}\right)\in\cal E_{1,\lambda}\left(\ynorm\right)\times\cal E_{2,\lambda}\left(\ynorm\right)\\
\mbox{such that: }\det\left(\tx D\Psi_{\lambda}\right)=1
\end{array}}\\
\left(\cal E_{1,-\lambda}\left(\ynorm\right)\times\cal E_{2,-\lambda}\left(\ynorm\right)\right)_{\omega} & := & \acc{\begin{array}{c}
\Psi_{-\lambda}=\left(\Psi_{1,-\lambda},\Psi_{2,-\lambda}\right)\in\cal E_{1,-\lambda}\left(\ynorm\right)\times\cal E_{2,-\lambda}\left(\ynorm\right)\\
\mbox{such that: }\det\left(\tx D\Psi_{-\lambda}\right)=1
\end{array}}.
\end{eqnarray*}
According to Proposition \ref{prop: description espace de module},
the map
\begin{eqnarray*}
\Lambda_{\pm\lambda}\left(\ynorm\right) & \longrightarrow & \cal E_{1,\pm\lambda}\left(\ynorm\right)\times\cal E_{2,\pm\lambda}\left(\ynorm\right)\\
\psi_{\pm\lambda} & \mapsto & \Psi_{\pm\lambda}:=\cal H_{\pm\lambda}\circ\psi_{\pm\lambda}\circ\cal H_{\pm\lambda}^{-1}
\end{eqnarray*}
given in (\ref{eq: isotropie dans espace des feuilles}) is a bijection
$\Big($we identify here $\Psi_{\pm\lambda}\left(x,\mathbf{h}\right)=\left(x,\Psi_{1,\pm\lambda}\left(\mathbf{h}\right),\Psi_{2,\pm\lambda}\left(\mathbf{h}\right)\right)$
with $\left(\Psi_{1,\pm\lambda}\left(\mathbf{h}\right),\Psi_{2,\pm\lambda}\left(\mathbf{h}\right)\right)$$\Big)$.
An easy computation based on (\ref{eq: integrales premieres}) gives:
\begin{eqnarray*}
\left(\cal H_{\pm\lambda}^{-1}\right)^{*}\left(\frac{\tx dy_{1}\wedge\tx dy_{2}}{x}\right) & = & \tx dh_{1}\wedge\tx dh_{2}+\ps{\tx dx}\,\,.
\end{eqnarray*}
This means in particular that $\psi_{\pm\lambda}$ is transversally
symplectic with respect to ${\displaystyle \omega=\frac{\tx dy_{1}\wedge\tx dy_{2}}{x}}$,
\emph{i.e.} 
\begin{eqnarray*}
\left(\psi_{\pm\lambda}\right)^{*}\left(\omega\right) & = & \omega+\ps{\tx dx}\,\,,
\end{eqnarray*}
if and only if $\Psi_{\pm\lambda}=\left(\Psi_{1,\pm\lambda},\Psi_{2,\pm\lambda}\right)$
preserves the standard symplectic form $\tx dh_{1}\wedge\tx dh_{2}$
in the space of leaves, \emph{i.e. $\det\left(\tx D\Psi_{\pm\lambda}\right)=1$}.
In other words:
\begin{prop}
We have the following bijections:
\begin{eqnarray*}
\Lambda_{\lambda}^{\omega}\left(\ynorm\right) & \tilde{\rightarrow} & \left(\cal E_{1,\lambda}\left(\ynorm\right)\times\cal E_{2,\lambda}\left(\ynorm\right)\right)_{\omega}\\
\psi_{\lambda} & \mapsto & \left(\Psi_{1,\lambda},\Psi_{2,\lambda}\right)
\end{eqnarray*}
and
\begin{eqnarray*}
\Lambda_{-\lambda}^{\omega}\left(\ynorm\right) & \tilde{\rightarrow} & \left(\cal E_{1,-\lambda}\left(\ynorm\right)\times\cal E_{2,-\lambda}\left(\ynorm\right)\right)_{\omega}\\
\psi_{-\lambda} & \mapsto & \left(\Psi_{1,-\lambda},\Psi_{2,-\lambda}\right)
\end{eqnarray*}
$\Big($we identify here $\Psi_{\pm\lambda}\left(x,\mathbf{h}\right)=\left(x,\Psi_{1,\pm\lambda}\left(\mathbf{h}\right),\Psi_{2,\pm\lambda}\left(\mathbf{h}\right)\right)$
with $\left(\Psi_{1,\pm\lambda}\left(\mathbf{h}\right),\Psi_{2,\pm\lambda}\left(\mathbf{h}\right)\right)$$\Big)$.
\end{prop}

\subsection{Quasi-linear Stokes phenomena in the first Painlevé equation}

~

In \cite{bittmann:tel-01367968}, we link the study of \emph{quasi-linear}
\emph{Stokes phenomena }(see \cite{Kapaev} for the first Painlevé
equation)\emph{ }to our Stokes diffeomorphisms. For instance, in the
case of the first Painlevé equation, we show that the quasi-linear
Stokes phenomena formula found by Kapaev in \cite{Kapaev} allows
to compute the terms $\Psi_{2,\lambda,0}\left(0\right)$ and $\Psi_{1,-\lambda,0}\left(0\right)$
in (\ref{eq: description moduli space}). More precisely, elementary
computations (using Kapaev's connection formula) give: 
\[
\Psi_{2,\lambda,0}\left(0\right)=i\Psi_{1,-\lambda,0}\left(0\right)=\frac{e^{\frac{i\pi}{8}}}{\sqrt{\pi}}2^{\frac{3}{8}}3^{\frac{1}{8}}\,\,.
\]
Moreover, our description of the Stokes diffeomorphisms implies a
more precise estimate of the order of the remaining terms in Kapaev's
formula. In a forthcoming paper, we will use the study of some \emph{non-linear
Stokes phenomena }for the second Painlevé equations\emph{ }(see \emph{e.g.
}\cite{ClarksonMcLeod}) in order to compute coefficients of the $\Psi_{i,\pm\lambda}$'s.

\bibliographystyle{alpha}
\bibliography{references_preprint_AIF}

\end{document}